\documentclass[11pt]{article}
\usepackage{amsfonts}
\usepackage{latexsym,amssymb,amsfonts,graphicx}
\usepackage{amsmath,dsfont}
\usepackage[linewidth=1pt]{mdframed}
\usepackage{mathrsfs}
\usepackage{epsfig}
\usepackage{tabularx}
\usepackage[margin=1in]{geometry} 
\usepackage{verbatim}

\usepackage{url}
\usepackage{bm}
\usepackage{color}
\usepackage{natbib}

\usepackage[unicode=true,colorlinks,citecolor=linkcolor,linkcolor=blue,urlcolor=blue]{hyperref}

\providecommand{\U}[1]{\protect\rule{.1in}{.1in}}

\newtheorem{thm}{Theorem}[section]

\newtheorem{definition}{Definition}[section]

\newtheorem{lem}{Lemma}[section]

\newtheorem{proposition}{Proposition}[section]
\newtheorem{rem}{Remark}[section]

\usepackage[bf,SL,BF]{subfigure}

\newenvironment{proof}[1][Proof]{\noindent\textbf{#1.} }{\ \rule{0.5em}{0.5em}}
\numberwithin{equation}{section}
\allowdisplaybreaks

\definecolor{linkcolor}{rgb}{0,0,0.7}
\definecolor{urlcolor}{rgb}{1,0,0}

\begin{document}

\title{On optimality of barrier dividend control under endogenous regime switching with application to Chapter 11 bankruptcy
}

\author{Wenyuan Wang\thanks{School of Mathematics and Statistics, Fujian Normal University, Fuzhou, 350007, P.R. China; and School of Mathematical Sciences, Xiamen University, Xiamen, 361005, P.R. China. Email: wwywang@xmu.edu.cn}
\and Xiang Yu\thanks{Department of Applied Mathematics, The Hong Kong Polytechnic University, Hung Hom,
Kowloon, Hong Kong. Email: xiang.yu@polyu.edu.hk}
\and Xiaowen Zhou\thanks{Department of Mathematics and Statistics, Concordia University, Canada. Email: xiaowen.zhou@concordia.ca}}

\date{\vspace{-4ex}}

\maketitle

\begin{abstract}
Motivated by recent developments in risk management based on the U.S. bankruptcy code, we revisit the De Finetti's optimal dividend problem by incorporating the reorganization process and regulator's intervention documented in Chapter 11 bankruptcy. The resulting surplus process, bearing financial stress towards the more subtle concept of bankruptcy, corresponds to a non-standard spectrally negative L\'{e}vy process with endogenous regime switching. Some explicit expressions of the expected present values under a barrier strategy, new to the literature, are established in terms of scale functions. With the help of these expressions, when the tail of the L\'evy measure is log-convex, the optimal dividend control is shown to be of the barrier type and the associated optimal barrier can be identified using scale functions of spectrally negative L\'evy processes. Some financial implications are also discussed in an illustrative example.
\end{abstract}

\ \\
\textbf{AMS}: 60G51, 91B05, 91G05, 93E20

\section{Introduction}
De Finetti's dividend optimization has always been an important topic in corporate finance and insurance, which effectively signals the financial health and stability of companies. This type of risk management by maximizing the expected present value (PV for short) of dividend payments has stimulated fruitful research in stochastic singular control and impulse control under various risk models. Some pioneering work can be found in \cite{DeFi1957}, \cite{Gerber1969}, \cite{SLG84}, \cite{JSh1995} among others. A spectrally negative L\'{e}vy process, referring to a L\'{e}vy process with purely downward jumps, has been popular in insurance applications with its capability of describing the surplus process that diffuses by collecting the premiums and jumps downside by claim payments. Some early works based on spectrally negative L\'{e}vy processes can be found in \cite{AvPP2007}, \cite{KypPal2007}, \cite{ReZ07}, \cite{Loe2008} and \cite{Loe2009}. In particular, \cite{AvPP2007}, \cite{Loe2008} and \cite{Loe2009} prove that the optimal dividend control is of the barrier type by using the fluctuation theory of spectrally negative L\'evy processes and the dynamic programming approach. For some comprehensive surveys on developments in optimal dividends and related methodology, we refer to \cite{ALT2009}, \cite{Av2009} and references therein. Recently, more new variations on optimal dividend problems have emerged by considering different risk models and control constraints, such as \cite{Bay13}, \cite{Bay14}, \cite{Cheung2017}, \cite{PYY18}, \cite{Re19}, \cite{AvLW2020}, \cite{Cheng2020}, \cite{DeA2020}, \cite{NPY2020}, \cite{JLYY2021}, \cite{AvLW2021}, \cite{Noba2021}, just to name a few.

On the other hand, the consideration of the liquidation process (Chapter 7 bankruptcy) and the reorganization process (Chapter 11 bankruptcy) in risk assessment and management based on the U.S. bankruptcy code has attracted more attention in past decades. As an important early contribution along this direction, \cite{BroCS2007} study Chapters 7 and 11 bankruptcy by proposing a three-barrier model of a firm whose capital structure contains risky debt, and address the problem of optimal debt and equity values. Within a similar framework of  \cite{BroCS2007}, \cite{LiTWZ2014} establish an explicit formula for the probability of liquidation when the surplus process follows a general time-homogeneous diffusion process. To fully capture some features in Chapters 7 and 11 bankruptcy, \cite{LiLTZ2020} recently introduce a piecewise time-homogeneous diffusion surplus process embedded with three barriers $a, b$ and $c$ ($a<b<c$) to model the surplus of the insurance company. The liquidation barrier a: once the surplus process down-crosses this barrier, the company ceases all operations and is liquidated due to its inability to cover its debts. The reorganization barrier b: once the surplus process down-crosses this barrier, 
an exponential clock starts that gives an amount of time until one up-crosses the barrier $c$. Moreover, the insurer's state becomes insolvent whose businesses will be reorganized subject to the regulator's intervention. The safety barrier c: an insurer whose surplus stays at or above this barrier is financially healthy. To reflect the influence by reoganization and the regulator's intervention, the dynamics of the surplus process switches between two different time-homogeneous diffusion processes. Other notable studies in risk management featuring Chapters 7 and 11 bankruptcy can be found in \cite{Paseka03}, \cite{BroKaya2007}, \cite{DaiJL2013}, \cite{CorDE2017}, etc.

As a first attempt to study the De Finetti's optimal dividend problem to integrate Chapter 11 bankruptcy, we assume that $a=-\infty$ and leave the more complicated model encoding both Chapters 7 and 11 bankruptcy (i.e., $a>-\infty$) as future research. We model the risk surplus process by a spectrally negative L\'{e}vy process with endogenous regime switching activated by a reorganization barrier $b$ and a solvency barrier $c$ ($b<c$). Thanks to the spatial homogeneity of of L\'{e}vy processes, it is assumed without loss of generality that $b=0$, which simplifies the presentation of some main results. The resulting risk surplus process is mathematically close to a spectrally negative L\'{e}vy process with regime switching if we classify the solvency and insolvency states as regime states. Nevertheless, our new risk surplus process differs substantially from the Markov additive models in the aspect that our endogenous regime switching is triggered when the controlled surplus process crosses some prescribed barriers in certain ways; see our detailed construction in Definition \ref{def001}. 
Meanwhile, 
our model allows jumps to capture some large shocks (for example, lump sums of claims) in the cash flow of the insurer's surplus level, which gives rise to some new mathematical challenges. Another key feature in Chapter 11 is its bankruptcy time that is defined as the first instance when the amount of time that the risk process continuously stays in the insolvency state exceeds an exponential grace time. This type of bankruptcy time is also called Parisian ruin with exponential delay motivated by the Parisian option, see \cite{Ches97}. For optimal dividend control with Parisian ruin and spectrally negative L\'{e}vy processes, we refer to \cite{Re19}. As explained in \cite{PPSY20}, the Parisian ruin with exponential delay is closely related to Poisson observation. Some recent results and applications on L\'{e}vy processes with Poisson observations can be found in \cite{ALIZ2016}, \cite{ALI17}, \cite{PPSY20} and references therein.

With both endogenous regime switching and the Parisian ruin time, it is an open problem that whether the optimal dividend control is still of the barrier type. We conjecture that the optimality of a barrier strategy still holds and perform the ``guess-and-verify'' procedure. We first introduce an indicator state process to capture the switching between solvency and insolvency states, leading to a pair of coupled PVs of dividend payments. By employing the fluctuation theory of spectrally negative L\'{e}vy processes and some perturbation arguments, we obtain some novel explicit formulas of expected PVs under a barrier strategy in terms of scale functions. Assuming that the tail of the L\'{e}vy measure is log-convex, we contribute the rigorous verification of the optimal barrier strategy using the HJB variational inequality and some delicate computations on generators and slope conditions with the aid of the expected PVs of dividend payments. Unlike the implicit fixed point barrier in \cite{NPY2020}, the optimal barrier in our model, depending on the solvency barrier $c$, can be characterized analytically; see its definition in \eqref{def.d*} and one illustrative example in Section \ref{sec:example}.

The rest of the paper is organized as follows. Section \ref{sec:singul} introduces the De Finetti's optimal singular dividend control under Chapter 11 bankruptcy, in which some explicit formulas are derived for the expected PVs under a barrier dividend strategy. Section \ref{sec:ver} constructs the candidate optimal barrier using scale functions of spectrally negative L\'{e}vy processes and verifies its optimality with the assistance of the HJB variational inequality and previous formulas of expected PVs of the barrier dividend. An illustrative example is presented in Section \ref{sec:example} and some financial implications are discussed. 

\section{De Finetti's Optimal Dividend under Chapter 11 Bankruptcy}\label{sec:singul}


\subsection{Problem formulation}
Write $(\Omega, \mathcal{F}, \mathbf{F}, \mathbb{P})$ for an underlying filtered probability space, where $\mathbf{F}=(\mathcal{F}_t)_{t\geq 0}$ stands for the natural filtration generated by two spectrally negative L\'evy processes $X(t)$ and $\widetilde{X}(t)$ and satisfies the usual conditions of right-continuity and completeness (see Exercise 8.10 in Chapter 8 of \cite{Kyp2014}).
Let us consider the {\it singular dividend control} $D=(D(t))_{t\geq 0}$, which is a non-decreasing and left-continuous $\mathbf{F}$-adapted process representing the cumulative dividends paid out by the company up to the Chapter 11 bankruptcy time. 
To account for changes in risk processes under regulator's intervention documented in Chapter 11, we introduce an auxiliary state process $I(t)$, depending on the control, as an indicator process of solvency and insolvency states. We adopt $X(t)$ and $\widetilde{X}(t)$ to model the underlying risk processes without and with regulator's interventions. To construct the surplus process $U(t)$, if the insurer is in the solvency state at time $t\geq0$ (i.e., $I(t)=0$), the surplus process
$U(t)$ is governed by a spectrally negative L\'evy process $X(t)$ deducted by total dividends that are only paid when the surplus level is at or above the safty barrier $c>0$. At the instant when the surplus process $U(t)$ down-crosses the level $b=0$, the state of the insurer switches to insolvency and the underlying risk process $X(t)$ switches to another spectrally negative L\'evy process $\widetilde{X}(t)$. On the other hand, if the insurer is in the insolvent state  at time $t\geq0$ (i.e., $I(t)=1$),
the surplus process follows the dynamics of the spectrally negative L\'evy process $\widetilde{X}(t)$. At the instant when the surplus process $U(t)$ up-crosses the safety barrier $c>0$, the state of the insurer switches to solvency state and the underlying risk process switches back to the process $X(t)$. It is assumed for mathematical tractability that the dividend is only paid when the surplus level is at least at the safety barrier $c>0$, which is also consistent with the real life situation that the regulator immediately notifies the insurer when the surplus level falls below the safety barrier.

We can summarize the previous piecewise construction of the surplus process $U$ under a singular dividend strategy in the next definition.

\begin{definition}\label{def001} \textbf{\emph{(The surplus process $U$).}}
Let $D=(D(t))_{t\geq 0}$ be a non-decreasing  left-continuous $\mathbf{F}$-adapted process, with its non-decreasing and continuous part $C_{D}(t)$ and $\triangle D(t):=D(t+)-D(t)$ for $ t\geq 0$. Recalling  $c>0$ and  starting at  time $\mathcal{T}_0:=0$, let $U(0):=X(0)\in\left(0,\infty\right)$ if $I(0)=0$ and $U(0):=\widetilde{X}(0)\in\left(-\infty,c\right)$ if $I(0)=1$. The process $(U, I)$ can be  constructed as follows:\vspace{0.2in}

\begin{mdframed}
{\small
\begin{itemize}
\item[\emph{(\textbf{a})}] Suppose that the process $\left(U,I\right)$ has been defined on $\left[0,\mathcal{T}_{n}\right]$ for some $n\geq 0$ with $\mathcal{T}_{n}<\infty$.
\begin{itemize}
    \item[$\bullet$]
If $I\left(\mathcal{T}_{n}\right)=0$, we define $U_{n+1}=(U_{n+1}(t))_{t\geq 0}$ with $U_{n+1}(0)=U\left(\mathcal{T}_{n}\right)$ according to the SDE
\begin{align}
\label{def4.1}
\mathrm{d}U_{n+1}(t)=
\mathrm{d}X(\mathcal{T}_{n}+t)
-\mathbf{1}_{\{U_{n+1}(t)-c\geq \triangle D(t)\}}\,\mathrm{d}D(t),\quad t\geq 0,
\end{align}
and update time $\mathcal{T}_{n+1}:=\mathcal{T}_{n}+\inf\{t\geq 0;\,U_{n+1}(t)<0\}$. We then define 
the process $\left(U,I\right)$ on $\left(\mathcal{T}_{n},\mathcal{T}_{n+1}\right]$ by
\begin{align}
\begin{cases}
\left(U\left(\mathcal{T}_{n}+t\right),I\left(\mathcal{T}_{n}+t\right)\right):=\left(U_{n+1}(t),0\right),\quad  t\in \left(0, \mathcal{T}_{n+1}-\mathcal{T}_{n}\right),
\vspace{0.2cm}\\
\left(U\left(\mathcal{T}_{n+1}\right),I\left(\mathcal{T}_{n+1}\right)\right):=\left(U_{n+1}\left(\mathcal{T}_{n+1}-\mathcal{T}_{n}\right),1\right).\end{cases}
\nonumber
\end{align}

\item[$\bullet$] Otherwise, if  $I\left(\mathcal{T}_{n}\right)=1$, we define $\widetilde{U}_{n+1}$  by
\begin{eqnarray}
\label{def4.2}
\widetilde{U}_{n+1}(t):=U\left(\mathcal{T}_{n}\right)+\widetilde{X}\left(\mathcal{T}_{n}+t\right)-\widetilde{X}\left(\mathcal{T}_{n}\right),\quad t\geq 0,
\end{eqnarray}
and update time $\mathcal{T}_{n+1}:=\mathcal{T}_{n}+\inf\{t\geq 0; \widetilde{U}(t)\geq c\}.$
We then define the process $\left(U,I\right)$ on $\left(\mathcal{T}_{n},\mathcal{T}_{n+1}\right]$ by
\begin{align}
\begin{cases}
\left(U\left(\mathcal{T}_{n}+t\right),I\left(\mathcal{T}_{n}+t\right)\right):=(\widetilde{U}_{n+1}\left(t\right),1),\quad t\in \left(0, \mathcal{T}_{n+1}-\mathcal{T}_{n}\right),
\vspace{0.2cm}\\
\left(U\left(\mathcal{T}_{n+1}\right),I\left(\mathcal{T}_{n+1}\right)\right):=(\widetilde{U}_{n+1}\left(\mathcal{T}_{n+1}-\mathcal{T}_{n}\right),0)=(c,0).\end{cases}
\nonumber
\end{align}
\end{itemize}
\item[\emph{(\textbf{b})}]
Suppose that the process $\left(U,I\right)$ has been defined on $\left[0,\mathcal{T}_{n}\right]$ for some $n\geq 0$ with $\mathcal{T}_{n}=\infty$, we update $\mathcal{T}_{n+1}=\infty$.
\end{itemize}
}
\end{mdframed}
\end{definition}

\ \hspace{0.3in}\includegraphics[height=2.5in]{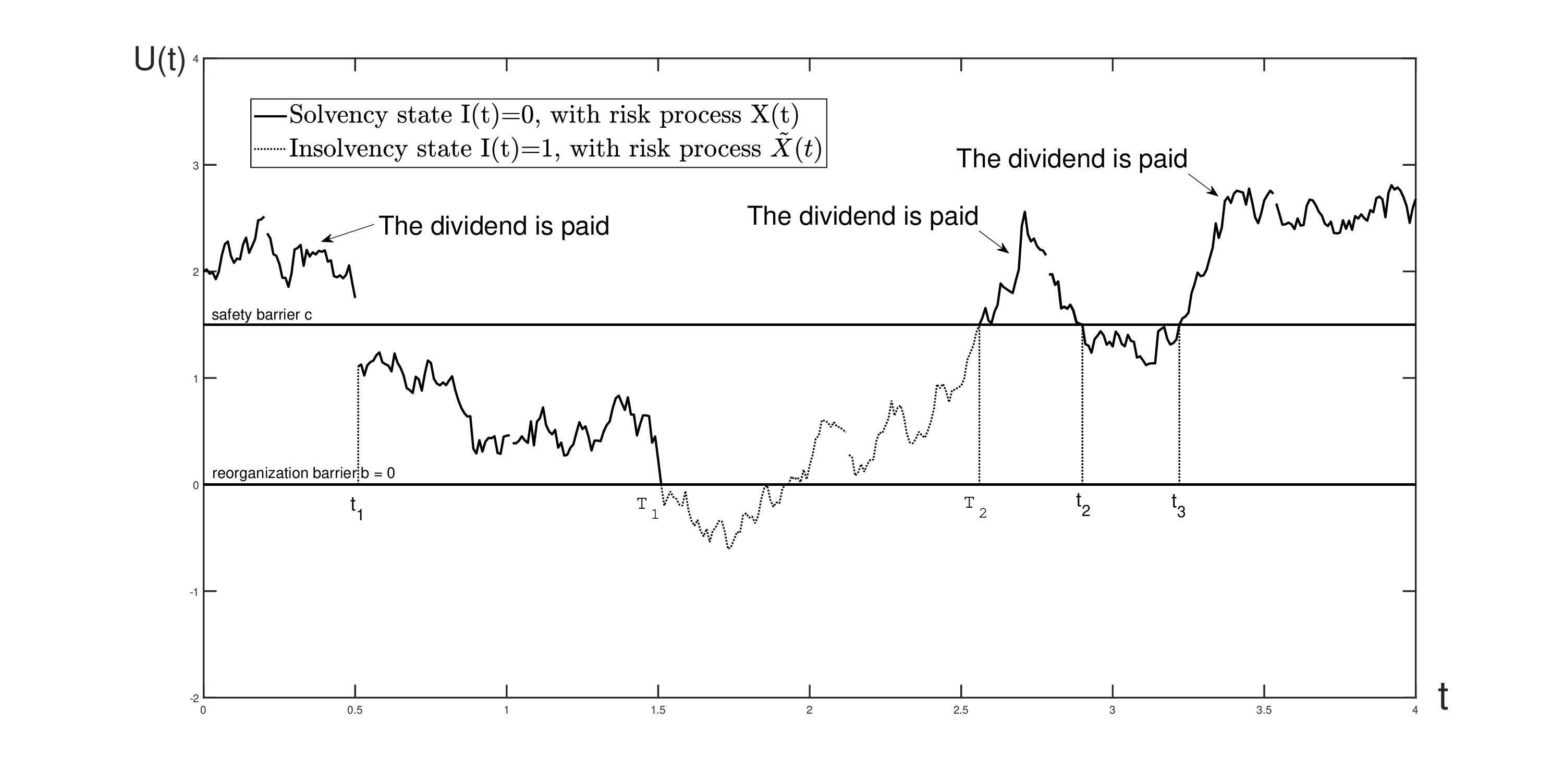}\\
\vbox{Figure 1: {\footnotesize A sample path of the surplus process $U(t)$ for $t\in[0,4]$: On the interval $[0,\mathcal{T}_1]$, the surplus level is in solvency state with $I(t)=0$ and $U(t)$ is generated by the risk process $X(t)$ and the dividend is paid during $[0,t_1)$; on the interval $(\mathcal{T}_1, \mathcal{T}_2]$, the surplus level is in insolvency state with $I(t)=1$ and $U(t)$ is generated by the risk process $\widetilde{X}(t)$; on the interval $(\mathcal{T}_2, 4]$, the surplus process switches back to solvency state again and the dividend is paid during $(\mathcal{T}_2, t_2]$ and $[t_3, 4]$.}}\vspace{0.2in}

Note that the one-dimensional process $U$ is not a Markov process in general, but the two-dimensional process $(U,I)$ is Markovian. Write $\mathbb{P}_{x,i}$ and $\mathbb{E}_{x,i}$ for the law of $(U,I)$ such that $U(0)=x$ and $I(0)=i$.

\begin{rem}
\label{rem.2.1.vamo}
    One can check that $\mathcal{T}_{n}\rightarrow\infty$ almost surely as $n\rightarrow\infty$. In fact, by the proof of Theorem 3.1 in \cite{ALIZ2016}, we recall the Laplace transform identity associated to one-sided exit problem that
\begin{eqnarray}
\label{two.side.l.}
\mathbb{E}_{x}\big[\mathrm{e}^{-q\widetilde{\tau}_{z}^{+}}\big]
=\exp(\widetilde{\Phi}_{q}(x-z)), \quad  x\leq z,
\end{eqnarray}
where $\widetilde{\tau}_{z}^{+}$ and $\widetilde{\Phi}_{q}$ are defined in the same manner as $\tau_{z}^{+}$ and $\Phi_{q}$ (of which the definitions are given in Section \ref{sec:pre}) but with the underlying process $X$ replaced as $\widetilde{X}$. 
Using \eqref{two.side.l.} one can get
    \begin{align}
    \label{u.b.L.}
        \mathbb{E}_{c,0}\left[\mathrm{e}^{-q \mathcal{T}_{2}}\mathbf{1}_{\{\mathcal{T}_{2}<\infty\}}\right]
        &= \mathbb{E}_{c,0}\left[\mathrm{e}^{-q \mathcal{T}_{1}}\mathbf{1}_{\{\mathcal{T}_{1}<\infty\}}
\mathbb{E}_{U(\mathcal{T}_{1}),1}\left[\mathrm{e}^{-q \mathcal{T}_{1}}\mathbf{1}_{\{\mathcal{T}_{1}<\infty\}}\right]\right]
\nonumber\\
 &= \mathbb{E}_{c,0}\left[\mathrm{e}^{-q \mathcal{T}_{1}}\mathbf{1}_{\{\mathcal{T}_{1}<\infty\}}
\mathbb{E}_{U(\mathcal{T}_{1})}\left[\mathrm{e}^{-q\widetilde{\tau}_{c}^{+}}\right]\right]
\nonumber\\
 &= \mathbb{E}_{c,0}\left[\mathrm{e}^{-q \mathcal{T}_{1}}\mathbf{1}_{\{\mathcal{T}_{1}<\infty\}}
\mathrm{e}^{\widetilde{\Phi}_{q}(U(\mathcal{T}_{1})-c)}\right]
\nonumber\\
 &\leq 
\mathrm{e}^{-\widetilde{\Phi}_{q}c}.
    \end{align}
where $U(\mathcal{T}_{1})\leq 0$ is used in the above inequality.
Then, by \eqref{u.b.L.}, one can obtain that
    \small{\begin{align}
        \mathbb{E}_{x,0}\left[\mathrm{e}^{-q \mathcal{T}_{2n}}\right]
        &=\mathbb{E}_{x,0}\left[\mathrm{e}^{-q \mathcal{T}_{2(n-1)}}
        \mathbf{1}_{\{\mathcal{T}_{2(n-1)}<\infty\}}\right.
        \nonumber\\&\quad 
        \left.
\times\mathbb{E}_{x,0}\left[\left.\mathrm{e}^{-q (\mathcal{T}_{2n}-\mathcal{T}_{2(n-1)})}
        \mathbf{1}_{\{\mathcal{T}_{2n}-\mathcal{T}_{2(n-1)}<\infty\}}
        \right|\mathcal{F}_{\mathcal{T}_{2(n-1)}}\right]\right]
        \nonumber\\
        &=
        \mathbb{E}_{x,0}\left[\mathrm{e}^{-q \mathcal{T}_{2(n-1)}}\mathbf{1}_{\{\mathcal{T}_{2(n-1)}<\infty\}}\right]\mathbb{E}_{c,0}\left[\mathrm{e}^{-q \mathcal{T}_{2}}\mathbf{1}_{\{\mathcal{T}_{2}<\infty\}}\right]
        \nonumber\\
        &=
        \mathbb{E}_{x,0}\left[\text{e}^{-q \mathcal{T}_{2}}\mathbf{1}_{\{\mathcal{T}_{2}<\infty\}}\right]\left[\mathbb{E}_{c,0}\left[\mathrm{e}^{-q \mathcal{T}_{2}}\mathbf{1}_{\{\mathcal{T}_{2}<\infty\}}\right]\right]^{n-1}
        \nonumber\\
        &\leq\mathrm{e}^{-(n-1)\widetilde{\Phi}_{q}c}\rightarrow 0, \quad \textit{as} \quad n\rightarrow\infty.
    \end{align}}Therefore, it holds that $\lim\limits_{n\rightarrow\infty} \mathcal{T}_{n} = \infty$ $\mathbb{P}_{x,0}$-a.s. because $\mathcal{T}_{n}$ is increasing in $n$. Similarly, one can also deduce that $\lim\limits_{n\rightarrow\infty} \mathcal{T}_{n} = \infty$ $\mathbb{P}_{x,1}$-a.s.
\end{rem}

By Remark \ref{rem.2.1.vamo}, one concludes that the process $U$ in Definition \ref{def001} is well-defined given that the stochastic differential equation \eqref{def4.1}  admits a unique strong solution for each $n\geq 0$.

\begin{definition}
A non-decreasing and left-continuous  $\mathbf{F}$-adapted process $D$ is said to be an admissible dividend strategy if the stochastic differential equation \eqref{def4.1}  admits a unique strong solution for each $n\geq 0$. Let $\mathcal{D}$ denote the set of all admissible dividend strategies.
\end{definition}

Similar to the Parisian ruin with exponential delay (see \cite{Re19}), we now give the definition of Chapter 11 bankruptcy time. Put $\kappa:=\inf\{k\geq 1: \mathcal{T}_k+e_{\lambda}^{k}<\mathcal{T}_{k+1},\ I(\mathcal{T}_k)=1\}$, with the convention that $\text{inf}\ \emptyset=\infty$. The Chapter 11 bankruptcy time is defined by
\begin{align}\label{11bankrcy}
T_D:=\left(\mathcal{T}_{\kappa}+e_{\lambda}^{\kappa}\right)\cdot\mathbf{1}_{\{\kappa<\infty\}}
+\left(+\infty\right)\cdot\mathbf{1}_{\{\kappa=\infty\}},
\end{align}
where $\{e_{\lambda}^{k}\}_{k\geq 1}$, defined on $(\Omega, \mathcal{F}, \mathbb{P})$, is a sequence of independent and  exponentially distributed random variables with parameter $\lambda$, representing the sequence of grace time periods granted by the regulator. It is assumed that $\{e_{\lambda}^{k}\}_{k\geq 1}$ is independent of $X$ and $\widetilde{X}$.

Recall that $\mathbb{P}_{x,0}$ is the probability law of $(U,I)$ with initial value $(U(0),I(0))=(x,0)$ for $x\in\left(0,\infty\right)$, $\mathbb{P}_{x,1}$ is the probability law of $(U,I)$ with initial value $(U(0),I(0))=(x,1)$ for $x\in\left(-\infty,c\right)$, and $\mathbb{E}_{x,0}$ (resp, $\mathbb{E}_{x,1}$) is the expectation operator under $\mathbb{P}_{x,0}$ (resp, $\mathbb{P}_{x,1}$). Given an admissible dividend strategy $D\in\mathcal{D}$, we consider two expected PVs of dividend payments defined by
{\small\begin{align}\label{expNPV-1}
&V_{D}(x):=\mathbb{E}_{x,0}\bigg[\int_{0}^{T_{D}}\mathrm{e}^{-qt}\mathbf{1}_{\{U(t)\geq c\}}\,\mathrm{d}C_{D}(t)
+\mathbf{1}_{\{U(t+)\geq c\}}\,\mathrm{d}\Big(\sum_{0\leq s\leq t}\triangle D(t)\Big)\bigg]
,\quad x>0,
\\
&\label{expNPV-2}
\widetilde{V}_{D}(x):=\mathbb{E}_{x,1}\bigg[\int_{0}^{T_{D}}\mathrm{e}^{-qt}\mathbf{1}_{\{U(t+)\geq c\}}\,\mathrm{d}D(t)\bigg],\quad x<c.
\end{align}}
%

In this paper, we want to find the optimal dividend strategy $D^{*}$ to attain the maximum of the value function that
\begin{align}\label{op.str.}
&V_{D^{*}}(x)=\sup_{D\in\mathcal{D}}V_{D}(x)\,\text{ for }\,\,x>0\quad
\text{ and } \quad \widetilde{V}_{D^{*}}(x)=\sup_{D\in\mathcal{D}}\widetilde{V}_{D}(x) \,\text{ for }\,\,x<c.
\end{align}

\subsection{Preliminaries on spectrally negative L\'evy processes}\label{sec:pre}
We conjecture that the optimal dividend control in problem \eqref{op.str.} is still of the barrier type and we aim to first express the expected PVs \eqref{expNPV-1} and \eqref{expNPV-2} under a barrier dividend strategy using scale functions of spectrally negative L\'{e}vy processes. To this end, we present here a brief review of some preliminary results on fluctuation identities for spectrally negative L\'{e}vy processes (see more details in \cite{Kyp2014}).
Let $X=(X(t))_{t\geq0}$ be a spectrally negative L\'{e}vy process defined on the filtered probability space $(\Omega,\mathbf{F},\mathbb{P})$ with the natural filtration $\mathbf{F}:=(\mathcal{F}_{t})_{t\geq0}$. 
To exclude trivial cases, it is assumed that $X$ has no monotone paths.
We denote by $\mathbb{P}_{x}$ the  probability law given $X(0)=x$, and by $\mathbb{E}_{x}$ the associated expectation. For ease of notation, we write $\mathbb{P}$ and $\mathbb{E}$ in place of $\mathbb{P}_{0}$ and $\mathbb{E}_{0}$ respectively. The Laplace transform of a spectrally negative L\'{e}vy process $X$ is defined by
$\mathbb{E}\big(\mathrm{e}^{\theta X(t)}\big)=\mathrm{e}^{t\psi(\theta)},$ for all $\theta\geq 0$, where
$$\psi(\theta)=\gamma\theta+\frac{1}{2}\sigma^2\theta^2+\int_{0}^{\infty}
(\mathrm{e}^{-\theta z}-1+\theta z\mathbf{1}_{(0,1]}(z))\upsilon(\mathrm{d}z),$$
for $\gamma\in(-\infty,\infty)$ and $\sigma\geq 0$, and the $\sigma$-finite L\'{e}vy measure $\upsilon$ of $X$ on $(0,\infty)$ satisfies that
$\int_{0}^{\infty}(1\wedge z^2)\upsilon(\mathrm{d}z)<\infty.$ 
As the Laplace exponent $\psi$ is strictly convex and $\lim_{\theta\rightarrow \infty}\psi(\theta)=\infty$, there exists a right inverse of $\psi$ defined by $\Phi_{q}:=\sup\{\theta\geq 0:\psi(\theta)=q\}$.

We next introduce some scale functions of $X$. For $q\geq0$, the scale function $W_{q}:\,[0,\infty)\rightarrow[0,\infty)$ is defined as the unique strictly increasing and continuous function on $[0,\infty)$ with Laplace transform $\int_{0}^{\infty}\mathrm{e}^{-\theta x}W_{q}(x)\mathrm{d}x=\frac{1}{\psi(\theta)-q}$, $\theta>\Phi_{q}$. For technical convenience, we extend the domain of $W_{q}(x)$ to the whole real line by setting $W_{q}(x)=0$ for $x<0$. Moreover, the scale functions $Z_{q}(x,\theta)$ and $ Z_{q}(x)$ are defined by
\begin{align}\label{Zqtheta}
Z_{q}(x,\theta)
:=\mathrm{e}^{\theta x}\Big(1-\left(\psi(\theta)-q\right)\int_{0}^{x}\mathrm{e}^{-\theta w}W_{q}(w)\mathrm{d}w\Big)
,\quad x\geq0,\,q\geq0,\,\theta\geq 0,
\end{align}
with $Z_{q}(x,\theta)=\mathrm{e}^{\theta x}$ on $(-\infty,0)$, and
$Z_{q}(x):=Z_{q}(x,0)$
with $Z_{q}(x)\equiv1$ on $(-\infty,0)$. We shall write $W:=W_{0}$ and $Z:=Z_{0}$ for simplicity.
It is well known that
\begin{align}\label{s.f.l.}
\lim_{x\rightarrow\infty}W_{q}^{\prime}(x)/W_{q}(x)=\Phi_{q} ,\quad \lim_{y\rightarrow\infty}W_{q}(x+y)/W_{q}(y)=\mathrm{e}^{\Phi_{q}x}.
\end{align}
Note that the scale function $W_{q}$ has right-hand and left-hand derivatives on $(0,\infty)$. When $X$ has bounded variation and the L\'{e}vy measure has no atoms or $X$ has unbounded variation, $W_{q}$ is continuously differentiable on $(0, \infty)$. By Theorems 3.10 and 3.12 in \cite{KKR12}, $W_{q}$ is twice continuously differentiable on $(0, \infty)$ when $X$ has a nontrivial Gaussian component or is $n+1$ times continuously differentiable on $(0, \infty)$ when $X$ has paths of bounded variation and the tail of the L\'evy measure (i.e., $\upsilon(x,\infty)$) is $n$ times continuously differentiable on $(0, \infty)$ and has a density that is dominated in order by $|x|^{-1-\alpha}$ in the neighbourhood of $0$ for some $\alpha>0$.
By Theorem 2 of \cite{Loe2009a}, $W_{q}$ is in $C^{\infty}(0, \infty)$ when the L\'evy measure has a completely monotone density.
We refer to \cite{KKR12}, \cite{Kyp2014}, \cite{ChanKS2011} and \cite{Loe2008} for more details on the regularity of scale functions.

For any $x\in\mathbb{R}$ and $\vartheta\geq0$, there exists a probability measure $\mathrm{P}_{x}^{\vartheta}$ obtained from the well-known exponential change of measure for a
spectrally negative L\'evy process such that
$
\frac{\mathrm{P}_{x}^{\vartheta}}{\mathrm{P}_{x}}\big|_{\mathcal{F}_{t}}=\mathrm{e}^{\vartheta\left(X(t)-x\right)-\psi(\vartheta)t}.
$
Under $\mathrm{P}_{x}^{\vartheta}$,  $X$ remains a spectrally negative L\'evy process with its Laplace exponent $\psi_{\vartheta}$ and the scale function $W^{\vartheta}_{q}$ that: for $\vartheta\geq 0$ and $q+\psi(\vartheta)\geq 0$,
\begin{equation}\label{c.s.newmeas.}
\psi_{\vartheta}(\theta)=\psi(\vartheta+\theta)-\psi(\vartheta)\,\,\,\text{ and }\,\,\,
W^{\vartheta}_{q}(x)=\mathrm{e}^{-\vartheta x}W_{q+\psi(\vartheta)}(x).
\end{equation}
In addition, denote by $W^{\vartheta}$ the $0$-scale function for $X$ under $\mathrm{P}_{x}^{\vartheta}$. For more detailed properties
concerning the exponential change of measure, we refer to Chapter 3 of \cite{Kyp2014}.

With the convention that  $\inf\emptyset=\infty$, the following notations of first passage times of the level $z\in (-\infty,\infty)$ by the process $X$ will be frequently used:
\begin{align}
\tau_{z}^{+}:=\inf\{t\geq 0: X(t)>z\} \,\,\, {\rm and} \,\,\, \tau_{z}^{-}:=\inf\{t\geq 0: X(t)<z\}.\nonumber
\end{align}
To model the underlying risk process when the regulator intervenes, let us also introduce another spectrally negative L\'evy process, denoted by $\widetilde{X}=\{\widetilde{X}(t);t\geq0\}$ with $(\widetilde{\gamma},\widetilde{\sigma},\widetilde{\upsilon})$,
$\widetilde{\psi}$ and $\widetilde{\Phi}_{q}$ representing its associated L\'evy triplet, the
Laplace exponent and the right inverse of Laplace exponent, respectively. Let $N$ (resp., $\widetilde{N}$), $\overline{N}$(resp., $\overline{\widetilde{N}}$), and $B$ (resp., $\widetilde{B}$) be, respectively, the Poisson random measure, the compensated Poisson random measure, and the Brownian motion of $X$ (resp., $\widetilde{X}$).

\subsection{Expected PVs of dividends with a barrier strategy}

To verify the optimality of a specific barrier strategy, we first consider the expected PVs of dividend payments in \eqref{expNPV-1} and \eqref{expNPV-2} under a barrier strategy with a barrier $d$. The construction of the piecewise underlying process $(U_d, I_d)$ follows Definition \ref{def001} by employing the barrier strategy. In particular, in item $(\textbf{a})$ of Definition \ref{def001}, the unique strong solution to the SDE \eqref{def4.1} can be constructed path by path that:\vspace{0.2in}
\begin{mdframed}
{\small
\begin{itemize}
\item If $I_{d}(\mathcal{T}_n)=0$, put
 $$X_{n+1}(t):=U_{d}(\mathcal{T}_n)+X(t+\mathcal{T}_n)-X(\mathcal{T}_n), \quad t\geq 0,$$ and define the process $U_{d}$ over $(\mathcal{T}_{n},\mathcal{T}_{n+1}]$ as
the process $X_{n+1}$ reflected from above at the level $d$, i.e.
\[U_{d}(\mathcal{T}_n+t):=X_{n+1}(t)-(\overline{X}_{n+1}(t)-d)\vee 0, \quad t\geq 0, \]
where $\overline{X}_{n+1}(t)=\sup_{0\leq s\leq t}X_{n+1}(s)$ denotes the running maximum process of $X_{n+1}$. Therefore, we have that $\mathcal{D}\neq\emptyset$.
\end{itemize}
}
\end{mdframed}\vspace{0.2in}
We shall consider two auxiliary moments of PVs of dividend payments that
\begin{align}\label{vn.tilvn}
\mathcal{V}_{n}(x):=\mathbb{E}_{x,0}\Big[\int_{0}^{T_{d}}\mathrm{e}^{-qt}\mathrm{d}D_{d}(t)\Big]^{n}
\quad \mbox{and} \quad
\widetilde{\mathcal{V}}_{n}(x):=\mathbb{E}_{x,1}\Big[\int_{0}^{T_{d}}\mathrm{e}^{-qt}
\mathrm{d}D_{d}(t)\Big]^{n},
\end{align}
where $n\geq 1$, and
\begin{align}\label{def.Dd}
D_{d}(t):=\sum_{n=1}^{\infty}\left[\left(\overline{X}_{n}\left(\mathcal{T}_{n}\wedge t-\mathcal{T}_{n-1}\right)-d\right)\vee 0\right]\mathbf{1}_{\{\mathcal{T}_{n-1}\leq t\}}
\mathbf{1}_{\{I_{d}(\mathcal{T}_{n-1})=0\}},\quad t\geq0,
\end{align}
representing total dividends paid on $[0,t]$, and we denote $T_{d}:=T_{D_{d}}$.
Let $U_{\infty}$ be the process $U_{d}$ with its dividend barrier $d=\infty$, i.e.,
$U_{\infty}(t):=\left.U_{d}(t)\right|_{d=\infty}$, $t\geq 0$, and no dividends are paid from $U_{\infty}$.
Denote by $\zeta _{z}^{\pm}(U_{\infty})$ and $T$, respectively, the first up(down)-crossing times of level $z$ and the Chapter 11 bankruptcy time for $U_{\infty}$, i.e.
\begin{align}\label{barzeta}
\quad \zeta _{z}^{\pm}(U_{\infty}):=\inf\{t\geq 0; U_{\infty}(t)\geq (<)z\}\quad \text{and}\quad T:=\left.T_{d}\right|_{d=\infty},
\end{align}
and we write $\zeta _{z}^{\pm}:=\zeta _{z}^{\pm}(U_{\infty})$ for simplicity.

We now introduce a key auxiliary function $\ell_{c}^{(q,\lambda)}(x)$ on $(-\infty,\infty)$ defined by
\begin{align}
\label{ell.def.}
\ell_{c}^{(q,\lambda)}(x)
&:=W_{q}(x)(1-
\mathrm{e}^{-\widetilde{\Phi}_{q+\lambda}c}Z_{q}(c,\widetilde{\Phi}_{q+\lambda}))
+\mathrm{e}^{-\widetilde{\Phi}_{q+\lambda}c}W_{q}(c)Z_{q}(x,\widetilde{\Phi}_{q+\lambda})
.
\end{align}
Note that $\ell_{c}^{(q,\lambda)}$ essentially plays the role of the scale function in future computations for our non-standard spectrally negative L\'{e}vy processes.

The next result gives some characterizations of two-sided exit problems for $U_{\infty}$.
\begin{lem}
\label{pro.01}
We have that
\begin{align}\label{twoside.exp.f}
\mathbb{E}_{x,0}\big[\mathrm{e}^{-q\zeta_{z}^{+}}\mathbf{1}_{\{\zeta_{z}^{+}<T\}}\big]
&=\ell_{c}^{(q,\lambda)}(x)/\ell_{c}^{(q,\lambda)}(z),\quad 0<x<z,\,\,c\leq z,
\\
\label{twoside.exp.02f}
\mathbb{E}_{x,1}\big[\mathrm{e}^{-q\zeta_{z}^{+}}\mathbf{1}_{\{\zeta_{z}^{+}<T\}}\big]
&=\mathrm{e}^{\widetilde{\Phi}_{q+\lambda}(x-c)}W_{q}(c)/\ell_{c}^{(q,\lambda)}(z),\quad x<c\leq z.
\end{align}
\end{lem}

\begin{proof}
Let $f(x):=\mathbb{E}_{x,0}\big[\mathrm{e}^{-q\zeta_{z}^{+}}; \zeta_{z}^{+}<T\big]$.
From Section 2 of \cite{ALIZ2016}, one can get the following well-known Laplace transform identities associated to two-sided exit problems that
\begin{align}
\label{two.side.e.}
\mathbb{E}_{x}\big[\mathrm{e}^{-q\tau_{z}^{+}}\mathbf{1}_{\{\tau_{z}^{+}<\tau_{0}^{-}\}}\big]
&=W_{q}(x)/W_{q}(z), \quad  x\leq z,
\\
\label{two.side.d.}
\quad \quad \mathbb{E}_{x}\big[\mathrm{e}^{-q\tau_{0}^{-}+\theta X(\tau_{0}^{-})}\mathbf{1}_{\{\tau_{0}^{-}<\tau_{z}^{+}\}}\big]
&=Z_{q}(x, \theta)-Z_{q}(z, \theta)W_{q}(x)/W_{q}(z), \quad  x\leq z.
\end{align}
Then, by \eqref{two.side.l.}, \eqref{two.side.e.} and \eqref{two.side.d.}, it is straightforward to check that
\small{\begin{align}\label{f.exp.f}
f(x)&=\mathbb{E}_{x,0}\big[\mathrm{e}^{-q\zeta_{z}^{+}}
\mathbf{1}_{\{\zeta_{z}^{+}<\zeta_{0}^{-}\}}\big]
+
\mathbb{E}_{x,0}\big[\mathrm{e}^{-q\zeta_{0}^{-}}\mathbf{1}_{\{\zeta_{0}^{-}<\zeta_{z}^{+}\}}
\mathbb{E}_{X(\zeta_{0}^{-}),1}\big[\mathrm{e}^{-q\zeta_{c}^{+}}\mathbf{1}_{\{\zeta_{c}^{+}<e_{\lambda}\}}
\big]f(c)\big]
\nonumber\\
&=
\mathbb{E}_{x}\Big[\mathrm{e}^{-q\tau_{z}^{+}}\mathbf{1}_{\{\tau_{z}^{+}<\tau_{0}^{-}\}}\Big]
+
\mathbb{E}_{x}\Big[\mathrm{e}^{-q\tau_{0}^{-}}\mathbf{1}_{\{\tau_{0}^{-}<\tau_{z}^{+}\}}
\mathbb{E}_{X(\tau_{0}^{-})}\big[\mathrm{e}^{-(q+\lambda)\widetilde{\tau}_{c}^{+}}\big]
\Big]f(c)
\nonumber\\
&=
\mathbb{E}_{x}\Big[\mathrm{e}^{-q\tau_{z}^{+}}\mathbf{1}_{\{\tau_{z}^{+}<\tau_{0}^{-}\}}\Big]
+
\mathbb{E}_{x}\Big[\mathrm{e}^{-q\tau_{0}^{-}}\mathbf{1}_{\{\tau_{0}^{-}<\tau_{z}^{+}\}}
\mathrm{e}^{\widetilde{\Phi}_{q+\lambda}(X(\tau_{0}^{-})-c)}
\Big]f(c)
\nonumber\\
&=
W_{q}(x)/W_{q}(z)+\mathrm{e}^{-\widetilde{\Phi}_{q+\lambda}c}
\big[Z_{q}(x,\widetilde{\Phi}_{q+\lambda})
-W_{q}(x)Z_{q}(z,\widetilde{\Phi}_{q+\lambda})/W_{q}(z)\big]f(c),
\end{align}}where $\widetilde{\tau}_{z}^{+}$ and $\widetilde{\Phi}_{q}$ are defined in the same manner of $\tau_{z}^{+}$ and $\Phi_{q}$ (of which the definitions are given in Section \ref{sec:pre}) but with the underlying process $X$ replaced by $\widetilde{X}$. 
In view of the expression in \eqref{ell.def.}, plugging $x=c$ back into \eqref{f.exp.f} yields that $f(c)=\frac{W_{q}(c)}{\ell_{c}^{(q,\lambda)}(z)}$, which combined with \eqref{f.exp.f} implies the desired result \eqref{twoside.exp.f}. In addition, the identity \eqref{twoside.exp.02f} follows easily from \eqref{two.side.l.}, \eqref{twoside.exp.f} as well as the facts that
\begin{eqnarray}
    \mathbb{E}_{x,1}\big[\mathrm{e}^{-q\zeta_{z}^{+}}\mathbf{1}_{\{\zeta_{z}^{+}<T\}}\big]
    =\mathbb{E}_{x,1}\big[\mathrm{e}^{-q\zeta_{c}^{+}}\mathbf{1}_{\{\zeta_{c}^{+}<e_{\lambda}\}}\big]\mathbb{E}_{c,0}\big[\mathrm{e}^{-q\zeta_{z}^{+}}\mathbf{1}_{\{\zeta_{z}^{+}<T\}}\big], \quad x<c\leq z,\nonumber
\end{eqnarray}
and $\ell_{c}^{(q,\lambda)}(c)=W_{q}(c)$.
\end{proof}

\begin{rem}
\label{rem01}
It follows from \eqref{twoside.exp.f} that the function $\ell_{c}^{(q,\lambda)}(x)$ is increasing on $(0,\infty)$. Indeed, due to the fact that $U$ has no positive jumps, we have $\mathbb{P}_{x,0}\left(\zeta_{y}^{+}>0\right)=1$ for $y>x$. As a result, it holds that
\begin{align*}
\ell_{c}^{(q,\lambda)}(x)/\ell_{c}^{(q,\lambda)}(y)=
\mathbb{E}_{x,0}[\mathrm{e}^{-q\zeta_{y}^{+}}\mathbf{1}_{\{\zeta_{y}^{+}<T\}}]<1 \Rightarrow \ell_{c}^{(q,\lambda)}(x)<\ell_{c}^{(q,\lambda)}(y),\quad c\leq x<y.
\end{align*}
In addition, we have the desired result that
\begin{align}
&\ell_{c}^{(q,\lambda)}(x)/\ell_{c}^{(q,\lambda)}(z)=\mathbb{E}_{x,0}[\mathrm{e}^{-q\zeta_{z}^{+}}\mathbf{1}_{\{\zeta_{z}^{+}<T\}}]
\leq
\mathbb{E}_{y,0}[\mathrm{e}^{-q\zeta_{z}^{+}}\mathbf{1}_{\{\zeta_{z}^{+}<T\}}]=\ell_{c}^{(q,\lambda)}(y)/\ell_{c}^{(q,\lambda)}(z)
\nonumber\\
&\Rightarrow \ell_{c}^{(q,\lambda)}(x)\leq \ell_{c}^{(q,\lambda)}(y), \quad 0<x<y<c<z.\nonumber
\end{align}
\end{rem}


Recall that the $n$-th moments of PV of dividends paid until Chapter 11 bankruptcy time are defined in \eqref{vn.tilvn} and \eqref{def.Dd}.
The next result gives important recursive formulas on moments of PVs of dividends. In particular, the expected PVs under a barrier strategy can be explicitly expressed in terms of scale functions; see similar results in \cite{ReZ07} for standard spectrally negative L\'{e}vy processes.

\begin{proposition}\label{thm3-1}
The $n$-th moment of PV of dividend payments until Chapter 11 bankruptcy time with the initial surplus $x=d$ admits the explicit form that
\begin{align}\label{nth.mom.}
\mathcal{V}_n(d)=n!\prod_{k=1}^{n}\ell_{c}^{(kq,\lambda)}(d)/\ell_{c}^{(kq,\lambda)\prime}(d).
\end{align}
Moreover, for arbitrary $x$, we obtain recursive equations that
\begin{align}\label{generalx-1}
\mathcal{V}_{n}(x)
&=
\frac{\ell_{c}^{(nq,\lambda)}(x)}{\ell_{c}^{(nq,\lambda)}(d)}
\mathcal{V}_{n}(d)\mathbf{1}_{(0,d]}(x)+\sum_{k=0}^{n}
 {n  \choose
  k}(x-d)^{k} \,\mathcal{V}_{n-k}(d)\mathbf{1}_{\left(d,\infty\right)}(x),
\\
\label{generalx-2}
\widetilde{\mathcal{V}}_{n}(x)
&=
\mathrm{e}^{\widetilde{\Phi}_{nq+\lambda}(x-c)}W_{nq}(c)\mathcal{V}_{n}(d)/\ell_{c}^{(nq,\lambda)}(d), \quad x\in\left(-\infty, c\right).
\end{align}
As a consequence, 
the expected PVs of dividend payments satisfy
\begin{align}
\label{V1.gen.}
\mathcal{V}_{1}(x)&=\ell_{c}^{(q,\lambda)}(x)/\ell_{c}^{(q,\lambda)\prime}(d)\mathbf{1}_{\left(0, d\right]}( x)+
\big[\ell_{c}^{(q,\lambda)}(d)/\ell_{c}^{(q,\lambda)\prime}(d)+x-d\big]\mathbf{1}_{\left(d,\infty\right)}(x),
\\
\label{V1.gen.001}
\widetilde{\mathcal{V}}_{1}(x)&=\mathrm{e}^{\widetilde{\Phi}_{q+\lambda}(x-c)}W_{q}(c)/\ell_{c}^{(q,\lambda)\prime}(d), \quad x\in\left(-\infty, c\right).
\end{align}
\end{proposition}

\begin{proof}
In view of $\zeta _{d + \varepsilon }^ +$ defined in \eqref{barzeta},
we claim that, for any positive integer $n\geq1$, the next two equations hold that
\begin{align} \label{expression1oflemma2}
{\mathbb{E}_{d,0}}\Big[ \Big[ {\int_0^{\zeta _{d +  \varepsilon }^ + } {{\mathrm{e}^{ - q s}}D_{d}(s)\mathrm{d}s} } \Big]^n{{\mathbf{1}}_{\{ {\zeta _{d +  \varepsilon }^ +  < {T_{d} }} \}}} \Big] &= o\left( { \varepsilon } \right),
\\ \label{expression2oflemma2}
{\mathbb{E}_{d,0}}\Big[ \Big[{\int_0^{{T_{d} }} {{\mathrm{e}^{ - q s}}\mathrm{d}D_{d}( s)} } \Big]^n{{\mathbf{1}}_{\{ {{T_{d} } < \zeta _{d +\varepsilon }^ + } \}}} \Big] &=  o\left(\varepsilon \right).
\end{align}
In fact, $U_{\infty}\left( {\zeta _{d + \varepsilon }^ + } \right) = d + \varepsilon $ implies $ D_{d}(s)\leq  \varepsilon$ for all $s\in[0,{\zeta _{d + \varepsilon }^ + }]$. Hence, due to $T_{d}\leq T$ almost surely, the left hand side of \eqref{expression1oflemma2} is less than
\begin{align}
&{ \varepsilon^n }{\mathbb{E}_{d,0}}\big[ \big[ \int_0^{\zeta _{d +  \varepsilon }^ + } \mathrm{e}^{ - q s}\mathrm{d}s  \big]^{n}\mathbf{1}_{\{ {\zeta _{d +  \varepsilon }^ +  < T} \}}\big]\nonumber\\
\leq&
\big[\varepsilon^{n}/q^{n}\big]\big[ {{\mathbb{E}_{d,0}}\big[ {{{\mathbf{1}}_{\{ {\zeta _{d + \varepsilon}^ +  < T} \}}}} \big] - {\mathbb{E}_{d,0}}\big[ {{\mathrm{e}^{ - q \zeta _{d + \varepsilon}^ + }}{{\mathbf{1}}_{\{ {\zeta _{d + \varepsilon}^ +  < T} \}}}} \big]} \big]
 \nonumber \\
=&\big[\varepsilon^{n}/q^{n}\big] \big[\ell_{c}^{(0,\lambda)}(d)/\ell_{c}^{(0,\lambda)}(d+\varepsilon)- \ell_{c}^{(q,\lambda)}(d)/\ell_{c}^{(q,\lambda)}(d+\varepsilon)\big]=o\left( { \varepsilon } \right), n\geq 1,\nonumber
\end{align}
which gives (\ref{expression1oflemma2}). In view of integration by parts and the fact that $\zeta_{\varepsilon}^{-} \leq \zeta_{0}^{-}(U_{d})\leq T_{d}$ almost surely on $\{T_{d}<\zeta_{d+\varepsilon}^{+}\}$, the left hand side of \eqref{expression2oflemma2} can be rewritten as
\begin{align}
&\mathbb{E}_{d,0}\big[ \big[\mathrm{e}^{ - q T_{d} }D_{d}( T_{d} ) + q \int_0^{{T_{d} }} \mathrm{e}^{ - q s}D_{d}(s)\mathrm{d}s\big]^{n}{\mathbf{1}_{\{ {{T_{d} } < \zeta _{d + \varepsilon }^ + } \}}} \big]\notag\\
 \leq &
 \mathbb{E}_{d,0}\big[ {\big[ \varepsilon \mathrm{e}^{ - q T_{d} }  + \varepsilon \int_0^{{T_{d} }}  {q\mathrm{e}^{ - q s}} \mathrm{d}s \big]^{n}{\mathbf{1}_{\{ {\zeta_{\varepsilon}^{-} < \zeta _{d +  \varepsilon }^ + } \}}}} \big]
\nonumber \\
= &\varepsilon^{n}\big[ Z(d-\varepsilon)-Z(d)W(d-\varepsilon)/W(d)\big]= o(\varepsilon),\quad n\geq 1,\nonumber
\end{align}
which verifies \eqref{expression2oflemma2}. It follows from \eqref{expression2oflemma2} that
\begin{align*}
\mathcal{V}_{k}\left( d \right) = \mathbb{E}_{d,0}\big[ \big[\int_{0}^{T_{d}}\mathrm{e}^{-qt}\mathrm{d}D_{d}(t)\big]^{k}\mathbf{1}_{\{\zeta _{d +\varepsilon }^{+}<T_{d}\}} \big]   + o(\varepsilon)
, \quad k\geq 1.
\end{align*}
Using the strong Markov property
and the Binomial Theorem, we can rewrite the first term on the right hand side of the above equation as
{\small\begin{align}
&\ \sum\limits_{i = 0}^k {k \choose i} \mathcal{V}_{k - i}( d)\,\mathbb{E}_{d,0}\big[ \big[ \int_0^{\zeta _{d +  \varepsilon }^ + } \mathrm{e}^{ - q s}\mathrm{d}D_{d}( s) \big]^i \mathrm{e}^{ - (k - i)q \zeta_{d +  \varepsilon }^ +}  \mathbf{1}_{\{ \zeta_{d +  \varepsilon }^ +  < T_{d} \}} \big]
\nonumber \\=&
\sum\limits_{i = 0}^k {k \choose i} \mathcal{V}_{k - i}(d)
\sum\limits_{j = 0}^i {i\choose j} \varepsilon^{j}\mathbb{E}_{d,0}\big[ \mathrm{e}^{ - q (k - i + j)\zeta_{d + \varepsilon}^+ } q ^{i - j}\big[ \int_0^{\zeta_{d + \varepsilon}^ + } \mathrm{e}^{ - q s}D_{d}(s)\mathrm{d}s\big]^{i - j}\mathbf{1}_{\{ \zeta_{d + \varepsilon}^ +  < T_{d} \}} \big].
\label{eqn:pass:-:3}
\end{align}}
Thanks to \eqref{expression1oflemma2}, \eqref{expression2oflemma2} and the fact that $T_{d}\leq T$ almost surely, it is sufficient to only consider the sum with $j=i=1$ or $j=i=0$ in (\ref{eqn:pass:-:3}) to get
{\small
\begin{align*}
\mathcal{V}_{k}\left( d \right)=& {\mathcal{V}_{k}}\left( d\right){\mathbb{E}_{d,0}}\big[ {{\mathrm{e}^{ - kq \zeta_{d +  \varepsilon }^ + }}{{\mathbf{1}}_{\{ {\zeta_{d +  \varepsilon }^ +  < {T_{d} }} \}}}} \big]+ k{\mathcal{V}_{k - 1}}\left( d\right){\mathbb{E}_{d,0}}\big[ { \varepsilon {\mathrm{e}^{ - kq \zeta_{d +  \varepsilon }^ + }}{{\mathbf{1}}_{\{ {\zeta_{d +  \varepsilon }^ +  < {T_{d} }} \}}}} \big]+ o\left( { \varepsilon } \right)
\nonumber\\
\leq&
\big({\mathcal{V}_{k}}(d) + k\varepsilon{\mathcal{V}_{k - 1}}(d)\big)
{\mathbb{E}_{d,0}}\big[ {{\mathrm{e}^{ - kq \zeta_{d +  \varepsilon }^ + }}{{\mathbf{1}}_{\{ {\zeta_{d +  \varepsilon }^ +  < T} \}}}} \big]
+ o\left( { \varepsilon } \right)
\nonumber \\
=&
\left({\mathcal{V}_{k}}(d) + k\varepsilon{\mathcal{V}_{k - 1}}(d)\right)
\ell_{c}^{(kq,\lambda)}(d)/\ell_{c}^{(kq,\lambda)}(d+\varepsilon)
+ o(\varepsilon),\nonumber
\end{align*}}which can be rearranged to
\begin{align}\label{eqn:pass:-:4}
&k\mathcal{V}_{k-1}(d) \,\ell_{c}^{(kq,\lambda)}(d)
\geq \mathcal{V}_{k}(d)\,
\big(\ell_{c}^{(kq,\lambda)}(d+\varepsilon)-\ell_{c}^{(kq,\lambda)}(d)
\big)/
\varepsilon + o(\varepsilon).
\end{align}
Letting $\varepsilon\rightarrow0$ in \eqref{eqn:pass:-:4}
, we get that
\begin{align}\label{diff.eq.}
k\mathcal{V}_{k-1}(d)\geq &\mathcal{V}_{k}(d)\,
\ell_{c}^{(kq,\lambda)\prime}(d)/\ell_{c}^{(kq,\lambda)}(d).
\end{align}
It remains to show that the reverse inequality of \eqref{diff.eq.} also holds. To this end, we consider the scenario that dividends are postponed until the level $d+\varepsilon$ is attained by the surplus process, at which instant a lump sum of amount $\varepsilon$ is paid as dividend and the dividend is again paid according to the barrier level $d$ afterwards. Due to the discounting, the aforementioned dividend payment yields a $k$-th moment of the discounted dividends being less than that of the barrier dividend strategy with the barrier $d$. As a result, it holds that
\begin{align*}
\mathcal{V}_{k}(d)
\geq&  \mathbb{E}_{d,0}\big[ \mathrm{e}^{ - kq \zeta_{d +  \varepsilon }^ + }\mathbf{1}_{\{ {\zeta_{d +  \varepsilon }^ +  < T} \}} \big]
\sum_{i=0}^{k}
 {k  \choose
  i}\varepsilon^{i} \,\mathcal{V}_{k-i}(d)\notag\\
=&
\big[\mathcal{V}_{k}(d) + k\varepsilon{\mathcal{V}_{k - 1}}(d)\big]
\ell_{c}^{(kq,\lambda)}(d)/\ell_{c}^{(kq,\lambda)}(d+\varepsilon) + o(\varepsilon),\nonumber
  \end{align*}
which leads to the reverse inequality in \eqref{diff.eq.} after taking the limit and rearranging the terms. Hence, $\mathcal{V}_{k}(d)$ satisfies the recursive formula in \eqref{nth.mom.}.

Finally, recalling that dividend payments are delayed until the first time when the surplus process up-crosses the barrier level $d$, we can derive by Markov property that \eqref{generalx-1} and \eqref{generalx-2} hold. As a direct result, the expected PVs under a barrier strategy satisfy the explicit expressions \eqref{V1.gen.} and \eqref{V1.gen.001}.
\end{proof}

\section{Optimal Barrier and Verification of the Optimality}\label{sec:ver}

For fixed $0<c<\infty$, with the convention $\sup\emptyset=-\infty$, let us define a candidate optimal barrier
\begin{align}\label{def.d*}
d^{*}:=\sup\{d\geq c: \ell^{(q,\lambda)\prime}_{c}(x)\geq \ell^{(q,\lambda)\prime}_{c}(d)\mbox{ for all }x\geq c\},
\end{align}
which is the largest value at which the function $d\mapsto\ell^{(q,\lambda)\prime}_{c}(d)$ on $[c,\infty)$ attains its minimum. It is conjectured that the barrier strategy with barrier level $d^{*}$ in \eqref{def.d*} solves the problem \eqref{op.str.}. We first present the next two lemmas as preparations.

\begin{lem}\label{lem001plus}
Suppose that $g_{1}$ and $g_{2}$ are continuous on $[a_{1},a_{2}]$, $g_{1}$ has both right-hand and left-hand derivatives on $(a_{1},a_{2})$, $g_{2}$ is continuously differentiable on $(a_{1},a_{2})$ with $g_{2}(a_{2})\neq g_{2}(a_{1})$ and $g_{2}^{\prime}(x)>0$ for $x\in (a_{1},a_{2})$.
Then there exists $\zeta\in(a_{1},a_{2})$ that
{\small\begin{align*}
\frac{\min\{g_{1}^{\prime-}(\zeta),g_{1}^{\prime+}(\zeta)\}}{g_{2}^{\prime}(\zeta)}\leq\frac{g_{1}(a_{2})-g_{1}(a_{1})}{g_{2}(a_{2})-g_{2}(a_{1})}\leq \frac{\max\{g_{1}^{\prime-}(\zeta),g_{1}^{\prime+}(\zeta)\}}{g_{2}^{\prime}(\zeta)},
\end{align*}}where $g_{1}^{\prime-}$ and $g_{1}^{\prime+}$ denote the right-hand and left-hand derivative of $g_{1}$, respectively.
\end{lem}

\begin{proof}
The proof is standard and hence omitted.
\end{proof}

Recall that a function $f$ defined on $(0,\infty)$ is  log-convex if the function $\log f$ is convex on $(0,\infty)$, and that the log-convexity implies convexity. 

\begin{lem}\label{lem002plus}
Suppose that $W_{q}^{\prime}(x)$ is log-convex. Then
$[W_{q}^{\prime}]^{\prime+}(x)/W_{q}^{\prime}(x)\uparrow\Phi_{q}$ and $[W_{q}^{\prime}]^{\prime-}(x)/W_{q}^{\prime}(x)\uparrow\Phi_{q}$ as $x\uparrow\infty.$
\end{lem}

\begin{proof}
By its convexity, $W_{q}^{\prime}(x)$ has right-hand and left-hand derivatives over $(0,\infty)$ with $\left[W_{q}^{\prime}\right]^{\prime\pm}(x)$ being non-decreasing, and
\begin{eqnarray}\label{dens.W'}
W_{q}^{\prime}(x)=W_{q}^{\prime}(x_{0})+\int_{x_{0}}^{x}[W_{q}^{\prime}]^{\prime\pm}(y)\,\mathrm{d}y,
\quad x\in[x_{0},\infty), \,\,x_{0}\in(0,\infty),
\end{eqnarray}
which together with $\lim_{x\rightarrow\infty}W_{q}(x)=\infty$ and \eqref{s.f.l.} implies the existence of a positive $x_{0}$ such that
$[W_{q}^{\prime}]^{\prime\pm}(x)>0$ for all $x\geq x_{0}$.
The log-convexity of $W_{q}^{\prime}(x)$ on $(0,\infty)$ implies that $[W_{q}^{\prime}]^{\prime\pm}(x)/W_{q}^{\prime}(x)$ is non-decreasing on $(0,\infty)$. It follows, together with the fact that
$W_{q}^{\prime}(x)$ is differentiable on $(0,\infty)$ except for countably many points, that
$[W_{q}^{\prime}]^{\prime-}(x)/W_{q}^{\prime}(x)
\leq [W_{q}^{\prime}]^{\prime+}(y)/W_{q}^{\prime}(y)
\leq
[W_{q}^{\prime}]^{\prime-}(z)/W_{q}^{\prime}(z)$ for $0<x<y<z.$
This inequality and the fact that $[W_{q}^{\prime}]^{\prime\pm}(x)>0$ for large $x$ lead to
\begin{align}\label{3.27rrl}
\lim\limits_{x\rightarrow\infty}[W_{q}^{\prime}]^{\prime-}(x)/W_{q}^{\prime}(x)
=\lim\limits_{x\rightarrow\infty}[W_{q}^{\prime}]^{\prime+}(x)/W_{q}^{\prime}(x)
\in (0,\infty].
\end{align}
Define two functions $F(x):=\frac{1}{W_{q}(1/x)}$ and $G(x):=\frac{1}{W_{q}^{\prime}(1/x)}$ for $x\in(0,\infty)$. In addition, let $F(0)=\lim\limits_{x\downarrow0}F(x)=0$ and $G(0)=\lim\limits_{x\downarrow0}G(x)=0$. Then  $F(x)$ is continuous over $[0,1]$, differentiable over $(0,1)$, strictly increasing, and  $F^{\prime}(x)>0$ for all $x>0$. Moreover, $G(x)$ is continuous over $[0,1]$, and has right-hand and left-hand derivatives on $(0,1)$.
By \eqref{s.f.l.}, Lemma \ref{lem001plus} and \eqref{3.27rrl}, we have that
\begin{align}
\frac{1}{\Phi_{q}}
=&\lim\limits_{x\downarrow 0}\frac{G(x)}{F(x)}=\lim\limits_{x\downarrow 0}\frac{G(x)-G(0)}{F(x)-F(0)}
\nonumber\\
\geq&\lim\limits_{x\downarrow 0}\frac{\min\{G^{\prime-}(\zeta),G^{\prime+}(\zeta)\}}{F^{\prime}(\zeta)}
= \lim\limits_{x\downarrow0}\frac{\min\{[W_{q}^{\prime}]^{\prime+}(1/\zeta),[W_{q}^{\prime}]^{\prime-}(1/\zeta)\}}{W_{q}^{\prime}(1/\zeta)} [\frac{W_{q}(1/\zeta)}{W_{q}^{\prime}(1/\zeta)}]^{2}
\nonumber\\
=&
\Phi_{q}^{-2}
\lim\limits_{x\rightarrow\infty}[W_{q}^{\prime}]^{\prime\pm}(x)/W_{q}^{\prime}(x),\nonumber
\end{align}
as well as the reverse inequality.
This, together with the non-decreasing property of $[W_{q}^{\prime}]^{\prime\pm}(x)/W_{q}^{\prime}(x)$, implies the desired result.
\end{proof}

The following Lemma \ref{thm4-2} gives the monotonicity of the function $d\mapsto\ell^{(q,\lambda)\prime}_{c}(d)$ on $[d^{*},\infty)$,  which plays an important role in solving our new dividend control problem under Chapter 11 bankruptcy. Before presenting this result, we recall that the tail of the L\'evy measure $\upsilon$ refers to the function $x\mapsto \upsilon(x,\infty)$ for $x\in(0,\infty)$.

\begin{lem}\label{thm4-2}
If the tail of the L\'evy measure $\upsilon$ is log-convex,
then the function $d\mapsto\ell^{(q,\lambda)\prime}_{c}(d)$ for $d\in[c,\infty)$ is increasing on $[d^{*},\infty)$.
\end{lem}

\begin{proof}
It follows from \eqref{ell.def.} that
\begin{align*}
&\ell^{(q,\lambda)\prime}_{c}(d)
=
W_{q}^{\prime}(d)[\psi(\widetilde{\Phi}_{q+\lambda})-q]\int_{0}^{c}\mathrm{e}^{-\widetilde{\Phi}_{q+\lambda} w}W_{q}(w)\mathrm{d}w+
\mathrm{e}^{-\widetilde{\Phi}_{q+\lambda}c}W_{q}(c)Z_{q}^{\prime}(d,\widetilde{\Phi}_{q+\lambda}).\nonumber
\end{align*}
We consider the following cases  separately.
\begin{itemize}
\item[(i)]  $\psi(\widetilde{\Phi}_{q+\lambda})-q>0$ (or, equivalently, $\widetilde{\Phi}_{q+\lambda}>\Phi_{q}$).

By Theorem 1.2 of \cite{LoeRe2010}, the function $W_{q}^{\prime}(x)$ is log-convex over $(0,\infty)$ when the tail of the L\'evy measure is log-convex. In view that $\widetilde{\Phi}_{q+\lambda}>\Phi_{q}$, and properties of scale functions in Section 2 of \cite{ALIZ2016} or Section 3.1 of \cite{Re19}, one has 
$$Z_{q}^{\prime}(x,\widetilde{\Phi}_{q+\lambda})=
[\psi(\widetilde{\Phi}_{q+\lambda})-q]\int_{0}^{\infty}
\mathrm{e}^{-\widetilde{\Phi}_{q+\lambda}y}W_{q}^{\prime}(x+y)\mathrm{d}y,\quad x\geq 0,$$
From Exercise 9 in Section 2.2 of \cite{Cons2010} (see also \cite{Artin1964}), it follows that the function $x\mapsto Z_{q}^{\prime}(x,\widetilde{\Phi}_{q+\lambda})$ is log-convex. Using Exercise 4 in Section 2.1 of \cite{Cons2010},
one deduces that the function $\ell^{(q,\lambda)\prime}_{c}(d)$ is log-convex on its domain $[c,\infty)$, and hence $\ell^{(q,\lambda)\prime}_{c}(d)$ is convex on $[c,\infty)$ as log-convexity implies convexity. This result and the definition of $d^{*}$ can guarantee that $\ell^{(q,\lambda)\prime}_{c}(d)$ is non-decreasing in $d$ on $[d^{*},\infty)$, as desired.
\item[(ii)] $\psi(\widetilde{\Phi}_{q+\lambda})-q<0$.

Using \eqref{Zqtheta} of $Z_{q}(x,\widetilde{\Phi}_{q+\lambda})$, the function $\ell^{(q,\lambda)\prime}_{c}(d)$ can be rewritten as
\begin{align}
\label{4.9.rev.}
&\ell^{(q,\lambda)\prime}_{c}(d)=
W_{q}^{\prime}(d)\big(\psi(\widetilde{\Phi}_{q+\lambda})-q\big)\int_{0}^{c}\mathrm{e}^{-\widetilde{\Phi}_{q+\lambda} w}W_{q}(w)\mathrm{d}w
\nonumber\\
&\quad+
\widetilde{\Phi}_{q+\lambda}
W_{q}(c)
\mathrm{e}^{\widetilde{\Phi}_{q+\lambda} (d-c)}\Big[1-\big(\psi(\widetilde{\Phi}_{q+\lambda})-q\big)\int_{0}^{d}\mathrm{e}^{-\widetilde{\Phi}_{q+\lambda} w}W_{q}(w)\mathrm{d}w\Big]
\nonumber\\
&\quad-[\psi(\widetilde{\Phi}_{q+\lambda})-q]\mathrm{e}^{-\widetilde{\Phi}_{q+\lambda}c}W_{q}(c)W_{q}(d).
\end{align}
Taking right-hand and left-hand derivatives on both sides of \eqref{4.9.rev.} with respect to $d$ gives that
\begin{align}
\label{g'}
&[\ell^{(q,\lambda)\prime}_{c}]^{\prime\pm}(d)=
\widetilde{\Phi}_{q+\lambda}\,\ell^{(q,\lambda)\prime}_{c}(d)-
\big(\psi(\widetilde{\Phi}_{q+\lambda})-q\big)W_{q}^{\prime}(d)\big[W_{q}(0)\nonumber\\
&\quad +\int_{0}^{c}\mathrm{e}^{-\widetilde{\Phi}_{q+\lambda} w}W_{q}^{\prime}(w)\mathrm{d}w-
\int_{0}^{c}\mathrm{e}^{-\widetilde{\Phi}_{q+\lambda} w}W_{q}(w)\mathrm{d}w\,[W_{q}^{\prime}]^{\prime\pm}(d)/W_{q}^{\prime}(d)
\big].
\end{align}
In addition, by \eqref{c.s.newmeas.} and integration by parts, it is also straightforward to verify that
\begin{align}\label{ell/W.1}
&W_{q}(0)+\int_{0}^{c}\mathrm{e}^{-\widetilde{\Phi}_{q+\lambda} w}W_{q}^{\prime}(w)\mathrm{d}w
-
\Phi_{q}\int_{0}^{c}\mathrm{e}^{-\widetilde{\Phi}_{q+\lambda} w}W_{q}(w)\mathrm{d}w
\nonumber\\
\geq&
[\Phi_{q}-\widetilde{\Phi}_{q+\lambda}]\Big[\frac{\mathrm{e}^{-\widetilde{\Phi}_{q+\lambda}c}W_{q}(c)}{\Phi_{q}-\widetilde{\Phi}_{q+\lambda}}
-W^{\Phi_{q}}(c)\int_{0}^{c}\mathrm{e}^{-\widetilde{\Phi}_{q+\lambda} w}
\mathrm{e}^{\Phi_{q} w}
\mathrm{d}w\Big]
\nonumber\\
=&
W^{\Phi_{q}}(c)>0,\quad \widetilde{\Phi}_{q+\lambda}<\Phi_{q}.
\end{align}
Hence, in view of \eqref{g'}, \eqref{ell/W.1}, as well as the facts of $[W_{q}^{\prime}]^{\prime\pm}(d)/W_{q}^{\prime}(d)
\uparrow\Phi_{q}$ as $d\uparrow\infty$ (see, Lemma \ref{lem002plus}) and $\ell^{(q,\lambda)\prime}_{c}(d)\geq 0$ on $[c,\infty)$ (see, Remark \ref{rem01}), we deduce that $\ell^{(q,\lambda)\prime}_{c}(d)$ is strictly increasing on $[c,\infty)$.
It follows that $d^{*}=c$ and $\ell^{(q,\lambda)\prime}_{c}(d)$ is strictly increasing in $d$ on $[d^{*},\infty)$.

\item[(iii)] $\psi(\widetilde{\Phi}_{q+\lambda})-q=0$.

Using \eqref{4.9.rev.} again, we derive that
$\ell^{(q,\lambda)\prime}_{c}(d)=
\widetilde{\Phi}_{q+\lambda}\mathrm{e}^{-\widetilde{\Phi}_{q+\lambda} c}W_{q}(c)
\mathrm{e}^{\widetilde{\Phi}_{q+\lambda} d}$, which is strictly increasing on $\left[c,\infty\right)$.
\end{itemize}

Putting all the pieces together, we conclude that $d^{*}\geq c$ and the function $d\mapsto\ell^{(q,\lambda)\prime}_{c}(d)$ is increasing on $[d^{*},\infty)$ in the case $\psi(\widetilde{\Phi}_{q+\lambda})-q>0$; and $d^{*}=c$ and the function $d\mapsto\ell^{(q,\lambda)\prime}_{c}(d)$ is increasing on $[d^{*},\infty)$ in the case $\psi(\widetilde{\Phi}_{q+\lambda})-q\leq 0$.
\end{proof}

\begin{rem}
It is important to note that, in Lemma \ref{thm4-2} and subsequent Lemmas \ref{lem3.4} and \ref{lem4-3}, and Theorems \ref{Lem4-1} and \ref{HJB4},  the log-convexity assumption is imposed only on the L\'evy measure $\upsilon$ of the spectrally negative L\'evy process $X$ (rather than on that of the spectrally negative L\'evy process $\widetilde{X}$). This is because the function $\ell^{(q,\lambda)}_{c}(x)$, which turns out to be the elementary object of study throughout this paper,
depends on $\widetilde{X}$ only through the single parameter $\widetilde{\Phi}_{q+\lambda}$, which allows us to use the log-convexity results from the literature to prove Lemma \ref{thm4-2}. For further discussions, see the up-coming Remark \ref{3.2.amo}.
\end{rem}

Define an operator $\mathcal{A}$ on the functional space $\mathcal{D}$ by
\begin{align}\label{operatorHJB}
\mathcal{A}f(x):=\gamma f^{\prime}(x)+\frac{\sigma^{2}}{2}f^{\prime\prime}(x)
+\int_{0}^{\infty}\big[f(x-z)-f(x)+f^{\prime}(x) z\mathbf{1}_{\{z\leq 1\}}\big]\upsilon(\mathrm{d}z),
\end{align}for $x\in\mathbb{R}$ and $f\in\mathcal{D}$. When $\sigma\in(0,\infty)$, $\mathcal{D}$ is the set of twice differentiable functions such that the integral on the right hand side of \eqref{operatorHJB} is finite;  when $\sigma=0$, $\mathcal{D}$ is the set of differentiable functions such that the integral on the right hand side of \eqref{operatorHJB} is finite.
Let $\widetilde{\mathcal{A}}$ be the operator similar to $\mathcal{A}$ where the L\'evy triplet $(\gamma,\sigma,\upsilon)$ of $X$ is replaced by $(\widetilde{\gamma},\widetilde{\sigma},\widetilde{\upsilon})$ of $\widetilde{X}$.

For fixed $y\in[c,\infty)$, we extend the definitions of $\mathcal{V}_{1}(x)$ and $\widetilde{\mathcal{V}}_{1}(x)$ with $d=y$ in \eqref{V1.gen.} and \eqref{V1.gen.001} and denote by $V_{y}(x)$
and $\widetilde{V}_{y}(x)$
the extended functions that
{\small\begin{align}\label{V_{y}(x)}
 V_{y}(x)&:=\frac
{\mathrm{e}^{\widetilde{\Phi}_{q+\lambda}(x-c)}W_{q}(c)\mathbf{1}_{\{x\leq0\}}}
{\ell^{(q,\lambda)\prime}_{c}(y)}+
\frac
{\ell^{(q,\lambda)}_{c}(x)\mathbf{1}_{\{0<x\leq y\}}}
{\ell^{(q,\lambda)\prime}_{c}(y)}
+\Big[\frac
{\ell^{(q,\lambda)}_{c}(y)}
{\ell^{(q,\lambda)\prime}_{c}(y)}+x-y\Big]\mathbf{1}_{\{x>y\}},
\\
\label{tildeV_{y}(x)}
\widetilde{V}_{y}(x)&:=
\frac{\mathrm{e}^{\widetilde{\Phi}_{q+\lambda}(x-c)}W_{q}(c)\mathbf{1}_{\{x\leq c\}}}{\ell_{c}^{(q,\lambda)\prime}(y)}
+
\frac
{\ell^{(q,\lambda)}_{c}(x)\mathbf{1}_{\{c<x\leq y\}}}
{\ell^{(q,\lambda)\prime}_{c}(y)}
+\Big[\frac
{\ell^{(q,\lambda)}_{c}(y)}
{\ell^{(q,\lambda)\prime}_{c}(y)}+x-y\Big]\mathbf{1}_{\{x>y\}}.
\end{align}}If the tail of the L\'evy measure $\upsilon$ is log-convex, $W_{q}(x)\in C^{1}(0,\infty)$ and $Z_{q}(x,\widetilde{\Phi}_{q+\lambda})\in C^{2}(0,\infty)$, and hence by \eqref{V_{y}(x)}, $V_{y}(x)$ is non-negative and $V_{y}(x)\in C^{2}(-\infty,0)\cap C^{1}(0,\infty)\cap C^{2}(y,\infty)$. 
By \eqref{tildeV_{y}(x)}, $\widetilde{V}_{y}$ is non-negative, increasing and twice continuously differentiable over $(-\infty,c)$.


\begin{lem}\label{lem3.4}
Suppose that the tail of the L\'evy measure $\upsilon$ is log-convex. The integral in \eqref{operatorHJB} with $f=V_{y}$
is finite for all $x\in(0,\infty)$ and is continuous for $x\in(0,y)$.
\end{lem}

\begin{proof}
We follow an argument similar to that of Lemma 4.1 in \cite{KypRS2010}. For $x\in(0,y)$, by \eqref{ell.def.} and \eqref{V_{y}(x)}, the integral in \eqref{operatorHJB} with $f=V_{y}$ can be decomposed as
{\small\begin{align}\label{Int.s}
\quad\quad\quad &
\frac{1-\mathrm{e}^{-\widetilde{\Phi}_{q+\lambda}c}Z_{q}(c,\widetilde{\Phi}_{q+\lambda})}
{\ell_{c}^{(q,\lambda)\prime}(y)}
%
\bigg[
\int_{0}^{\frac{(y-x)\wedge x}{2}\wedge 1}\big[W_{q}(x-z)-W_{q}(x)
+zW_{q}^{\prime}(x)\big]\upsilon(\mathrm{d}z)
\nonumber\\
&
+\int_{\frac{(y-x)\wedge x}{2}\wedge1}^{x}\big[W_{q}(x-z)-W_{q}(x)
+zW_{q}^{\prime}(x)\mathbf{1}_{\{z\leq 1\}}\big]\upsilon(\mathrm{d}z)
\bigg]
\nonumber\\
&
+\frac{\mathrm{e}^{-\widetilde{\Phi}_{q+\lambda}c}W_{q}(c)}{\ell_{c}^{(q,\lambda)\prime}(y)}
\bigg[
\int_{0}^{\frac{(y-x)\wedge x}{2}\wedge 1}\hspace{-0.2cm}[Z_{q}(x-z,\widetilde{\Phi}_{q+\lambda})-Z_{q}(x,\widetilde{\Phi}_{q+\lambda})
+zZ_{q}^{\prime}(x,\widetilde{\Phi}_{q+\lambda})]\upsilon(\mathrm{d}z)
\nonumber\\
&
+\int_{\frac{(y-x)\wedge x}{2}\wedge1}^{x}\big[Z_{q}(x-z,\widetilde{\Phi}_{q+\lambda})-Z_{q}(x,\widetilde{\Phi}_{q+\lambda})
+zZ_{q}^{\prime}(x,\widetilde{\Phi}_{q+\lambda})\mathbf{1}_{\{z\leq 1\}}\big]\upsilon(\mathrm{d}z)
\bigg]
\nonumber\\
&
+\frac{1}{\ell_{c}^{(q,\lambda)\prime}(y)}
\int_{x}^{\infty}\big[\mathrm{e}^{\widetilde{\Phi}_{q+\lambda}(x-z-c)}W_{q}(c)
-\ell_{c}^{(q,\lambda)}(x)
+z\ell_{c}^{(q,\lambda)\prime}(x)\mathbf{1}_{\{z\leq 1\}}\big]\upsilon(\mathrm{d}z).
\end{align}}When the tail of the L\'evy measure is log-convex, $W_{q}^{\prime}(x)$ is continuous and log-convex on $(0,\infty)$. It follows that \eqref{dens.W'} holds and $Z_{q}(x,\widetilde{\Phi}_{q+\lambda})\in C^{2}(0,\infty)$. As a result, by mean value theorem, we can obtain five upper bounds
for the five integrands in \eqref{Int.s}. 
Combining these bounds and the fact of $\int_{0}^{\infty}(1\wedge z^{2})\upsilon(\mathrm{d}z)<\infty$, we get that all five integrals in \eqref{Int.s} are finite. Using the dominated convergence theorem and the same upper bounds, one can further conclude that all five integrals in \eqref{Int.s} are continuous in $x\in(0,y)$. In addition, it is straightforward to show that the integral is finite for $x\in[y,\infty)$, which completes the proof.
\end{proof}

The next result gives the PDE characterization of the optimal value function.

\begin{thm}
[Verification Theorem]\label{Lem4-1}
Suppose that the tail of the L\'evy measure $\upsilon$ is log-convex. Assume that $V_{y}$ and $\widetilde{V}_{y}$ are non-decreasing functions such that
\begin{align}\label{hjbine4}
\left\{
\begin{aligned}
&(\mathcal{A}-q)V_{y}(x)\leq 0, & \quad x\in \left(0, \infty\right), \\
& (\widetilde{\mathcal{A}}-q-\lambda)\widetilde{V}_{y}(x)\leq 0,& \quad x\in\left(-\infty,c\right),
\\
& V_{y}^{\prime}(x)\geq 1,& \quad x\in\left(c,\infty\right).
\end{aligned}
\right.
\end{align}
Then, we have  
$V_{y}(x)\geq V_{D}(x)$ for $x\in\left[0,\infty\right)$, and $\widetilde{V}_{y}(x)\geq \widetilde{V}_{D}(x)$ for $x\in\left(-\infty,c\right)$ for any admissible singular dividend strategy $D\in \mathcal{D}$.\end{thm}

\begin{proof}
Recall from Remark \ref{rem.2.1.vamo} that $\mathcal{T}_{n}\uparrow\infty$ almost surely under both $\mathbb{P}_{x,0}$ and $\mathbb{P}_{x,1}$ as $n\uparrow\infty$.
For a given dividend strategy $D\in\mathcal{D}$ and the resulting surplus process $U$ given by Definition \ref{def001}, we denote $(C_{D}(t))_{t\geq0}$ 
the continuous part of $(D(t))_{t\geq0}$. 
Let $\zeta_{n}:=\inf\{t\geq0; |U(t)|\geq n\}$ for $n\geq 1$, it holds that $\zeta_{n}\uparrow\infty$ a.s. as $n\rightarrow\infty$. In addition, $U(t-)$ is bounded in the compact set $\left[-n,n\right]$ for $t<\zeta_{n}$.
We define
\begin{align*}
&\mathbb{S}:= \cup_{k\geq0,I(\mathcal{T}_{k})=0}\left[\mathcal{T}_{k}, \mathcal{T}_{k+1}\right),\,
\overline{\mathbb{S}}:= \cup_{k\geq0,I(\mathcal{T}_{k})=1}\left[\mathcal{T}_{k}, \mathcal{T}_{k+1}\right)
, \,J(t):=qt+\lambda\int_{0}^{t}\mathbf{1}_{\overline{\mathbb{S}}}(s)\,\mathrm{d}s,\nonumber\\
&\mathcal{G}(U,I,\mathbb{S},\overline{\mathbb{S}})(s):=\left(\mathcal{A}-q\right)V_{y}(U(s))\,\mathbf{1}_{\mathbb{S}}(s)
+(\widetilde{\mathcal{A}}-q-\lambda)\widetilde{V}_{y}(U(s))\,\mathbf{1}_{\overline{\mathbb{S}}}(s),\nonumber
\end{align*}
where $\mathbb{S}$ (resp., $\overline{\mathbb{S}}$) represents the union of all solvency (resp., insolvency) time intervals.
Note that $V_{y}(x)=\widetilde{V}_{y}(x)$ for all $x\in\left(-\infty,0\right)\cup\{c\}$ and that $V_{y}(0)=\widetilde{V}_{y}(0)$ when $X$ has paths of unbounded variation. It follows that
\begin{itemize}
    \item if $I(0)=0$, for $i\geq1$ such that $\mathcal{T}_{2i-2}<\infty$, define $\zeta_{1/m}^{(2i-1)-}:=\inf\{t\geq \mathcal{T}_{2i-2}; U(t)\leq 1/m\}$, $\xi_{1/m,m}^{(2i-1)}:=\inf\{t\geq 0; U_{2i-1}(t)\geq m \text{ or }U_{2i-1}(t)\leq 1/m\}$ where $U_{2i-1}$ is defined in \eqref{def4.1} with $n=2i-2$, and
        $\xi_{-m,c}^{(2i)}:=\inf\{t\geq0; \widetilde{U}_{2i}(t)\geq c\text{ or }\widetilde{U}_{2i}(t)\leq -m\}$ where $\widetilde{U}_{2i}$ is defined in \eqref{def4.2} with $n=2i-1$. By the non-decreasing property of $V_{y}$ and $U(\mathcal{T}_{2i-1})\leq 0<U(\zeta_{1/m}^{(2i-1)-}-)$ 
        when $\mathcal{T}_{2i-1}<\infty$, we have
\begin{align}\label{4.13}
&\left[\mathrm{e}^{-J(\zeta_{m})}V_{y}\left(U(\zeta_{m})\right)-V_{y}(x)\right]
\mathbf{1}_{\left[\mathcal{T}_{2k},\mathcal{T}_{2k+1}\right)}(\zeta_{m})
%
\nonumber\\
=&
\mathbf{1}_{\left[\mathcal{T}_{2k},\mathcal{T}_{2k+1}\right)}(\zeta_{m})\Big[
\mathrm{e}^{-J(\zeta_{m})}V_{y}\left(U(\zeta_{m})\right)
-\mathrm{e}^{-J(\mathcal{T}_{2k})}V_{y}(U(\mathcal{T}_{2k}))\notag\\
&+
\sum_{i=1}^{k}\left[\mathrm{e}^{-J(\mathcal{T}_{2i})}\widetilde{V}_{y}(U(\mathcal{T}_{2i}))
-\mathrm{e}^{-J(\mathcal{T}_{2i-1})}\widetilde{V}_{y}(U(\mathcal{T}_{2i-1}))
\right]
\nonumber\\
&+
\sum_{i=1}^{k}\left[\mathrm{e}^{-J(\mathcal{T}_{2i-1})}V_{y}(U(\mathcal{T}_{2i-1}))
-\mathrm{e}^{-J(\mathcal{T}_{2i-2})}V_{y}(U(\mathcal{T}_{2i-2}))
\right]
\Big]
\nonumber\\
\leq&
\mathbf{1}_{\left[\mathcal{T}_{2k},\mathcal{T}_{2k+1}\right)}(\zeta_{m})\Big[
\mathrm{e}^{-J(\zeta_{m})}V_{y}\left(U(\zeta_{m})\right)
-\mathrm{e}^{-J(\mathcal{T}_{2k})}V_{y}(U(\mathcal{T}_{2k}))\notag\\
&+
\sum_{i=1}^{k}\left[\mathrm{e}^{-J(\mathcal{T}_{2i})}\widetilde{V}_{y}(U(\mathcal{T}_{2i}))
-\mathrm{e}^{-J(\mathcal{T}_{2i-1})}\widetilde{V}_{y}(U(\mathcal{T}_{2i-1}))
\right]
\nonumber\\
&+
\sum_{i=1}^{k}[\mathrm{e}^{-J(\zeta_{1/m}^{(2i-1)-}-)}V_{y}(U(\zeta_{1/m}^{(2i-1)-}-))
-\mathrm{e}^{-J(\mathcal{T}_{2i-2})}V_{y}(U(\mathcal{T}_{2i-2}))
]
\Big],
\end{align}
and
\begin{align}
\label{4.14}
&\left[\mathrm{e}^{-J(\zeta_{m})}\widetilde{V}_{y}\left(U(\zeta_{m})\right)-
V_{y}(x)\right]\mathbf{1}_{\left[\mathcal{T}_{2k+1},\mathcal{T}_{2k+2}\right)}(\zeta_{m})
\nonumber\\
\leq&
\mathbf{1}_{[\mathcal{T}_{2k+1},\mathcal{T}_{2k+2})}(\zeta_{m})\Big[\mathrm{e}^{-J(\zeta_{m})}
\widetilde{V}_{y}(U(\zeta_{m}))
-\mathrm{e}^{-J(\mathcal{T}_{2k+1})}\widetilde{V}_{y}(U(\mathcal{T}_{2k+1}))\notag\\
&+
\sum_{i=1}^{k+1}\left[
\mathrm{e}^{-J(\zeta_{1/m}^{(2i-1)-}-)}V_{y}(U(\zeta_{1/m}^{(2i-1)-}-))
-\mathrm{e}^{-J(\mathcal{T}_{2i-2})}V_{y}(U(\mathcal{T}_{2i-2}))
\right]
\nonumber\\
&+\sum_{i=1}^{k}\left[\mathrm{e}^{-J(\mathcal{T}_{2i})}\widetilde{V}_{y}(U(\mathcal{T}_{2i}))
-\mathrm{e}^{-J(\mathcal{T}_{2i-1})}\widetilde{V}_{y}(U(\mathcal{T}_{2i-1}))
\right]
\Big].
\end{align}
\item if $I(0)=1$,
for $i\geq1$ such that $\mathcal{T}_{2i-1}<\infty$, define $\zeta_{1/m}^{(2i)-}:=\inf\{t\geq \mathcal{T}_{2i-1}; U(t)\leq 1/m\}$
.
By the non-decreasing property of $V_{y}$ and $U(\mathcal{T}_{2i})\leq 0<U(\zeta_{1/m}^{(2i)-}-)$ 
when $\mathcal{T}_{2i}<\infty$, we have
\begin{align}
\label{4.13.01}
&
\left[\mathrm{e}^{-J(\zeta_{m})}V_{y}\left(U(\zeta_{m})\right)-\widetilde{V}_{y}(x)\right]
\mathbf{1}_{\left[\mathcal{T}_{2k+1},\mathcal{T}_{2k+2}\right)}(\zeta_{m})
\nonumber\\
\leq&
\mathbf{1}_{\left[\mathcal{T}_{2k+1},\mathcal{T}_{2k+2}\right)}(\zeta_{m})\Big[\mathrm{e}^{-J(\zeta_{m})}V_{y}\left(U(\zeta_{m})\right)
-\mathrm{e}^{-J(\mathcal{T}_{2k+1})}V_{y}(U(\mathcal{T}_{2k+1}))\notag\\
&+
\sum_{i=0}^{k}\left[
\mathrm{e}^{-J(\mathcal{T}_{2i+1})}\widetilde{V}_{y}(U(\mathcal{T}_{2i+1}))
-\mathrm{e}^{-J(\mathcal{T}_{2i})}\widetilde{V}_{y}(U(\mathcal{T}_{2i}))
\right]
\nonumber\\
&+\sum_{i=1}^{k}\left[\mathrm{e}^{-J(\zeta_{1/m}^{(2i)-}-)}
V_{y}(U(\zeta_{1/m}^{(2i)-}-))
-\mathrm{e}^{-J(\mathcal{T}_{2i-1})}V_{y}(U(\mathcal{T}_{2i-1}))
\right]
\Big],
\end{align}
and
\begin{align}
\label{4.14.01}
&\left[\mathrm{e}^{-J(\zeta_{m})}\widetilde{V}_{y}\left(U(\zeta_{m})\right)-
\widetilde{V}_{y}(x)\right]\mathbf{1}_{\left[\mathcal{T}_{2k},\mathcal{T}_{2k+1}\right)}(\zeta_{m})
\nonumber\\
\leq&
\mathbf{1}_{\left[\mathcal{T}_{2k},\mathcal{T}_{2k+1}\right)}(\zeta_{m})\Big[
\mathrm{e}^{-J(\zeta_{m})}\widetilde{V}_{y}\left(U(\zeta_{m})\right)
-\mathrm{e}^{-J(\mathcal{T}_{2k})}\widetilde{V}_{y}(U(\mathcal{T}_{2k}))\notag\\
&+
\sum_{i=1}^{k}\left[\mathrm{e}^{-J(\zeta_{1/m}^{(2i)-}-)}
V_{y}(U(\zeta_{1/m}^{(2i)-}-))
-\mathrm{e}^{-J(\mathcal{T}_{2i-1})}V_{y}(U(\mathcal{T}_{2i-1}))
\right]
\nonumber\\
&+
\sum_{i=1}^{k}\left[\mathrm{e}^{-J(\mathcal{T}_{2i-1})}\widetilde{V}_{y}(U(\mathcal{T}_{2i-1}))
-\mathrm{e}^{-J(\mathcal{T}_{2i-2})}\widetilde{V}_{y}(U(\mathcal{T}_{2i-2}))
\right]
\Big].
\end{align}
\end{itemize}
By the proof of Lemma \ref{thm4-2}, $\ell^{(q,\lambda)\prime}_{c}(x)$ is either increasing over $(0,\infty)$ or is initially decreasing and ultimately increasing over $(0,\infty)$, implying that $V_{y}(x)$ can be written as the difference of two convex functions.
Therefore, one can apply the Meyer-It\^{o} formula (see, Theorem 70
of Chapter IV in \cite{Protter05}) to the processes
$\big\{\mathrm{e}^{-qt}V_{y}(U_{2i-1}(t))\big\}$ and $\{\mathrm{e}^{-(q+\lambda)t}\widetilde{V}_{y}(\widetilde{U}_{2i}(t))\}$ (resp., $\big\{\mathrm{e}^{-(q+\lambda)t}\widetilde{V}_{y}(\widetilde{U}_{2i-1}(t))\big\}$ and $\{\mathrm{e}^{-qt}V_{y}(U_{2i}(t))\}$) with $U_{n+1}(t)$ and $\widetilde{U}_{n+1}(t)$ in Definition \ref{def001} for $n=2i-2$ and $2i-1$ 
when $I(0)=0$ (resp., I(0)=1) to expand the right hand sides of \eqref{4.13}, \eqref{4.14}, \eqref{4.13.01} and \eqref{4.14.01}, and sum them over $k\geq0$ to obtain that
{\small
\begin{align}\label{4.18}
&\mathrm{e}^{-J(\zeta_{m})}V_{y}(U(\zeta_{m}))\mathbf{1}_{\mathbb{S}}(\zeta_{m})
+\mathrm{e}^{-J(\zeta_{m})}\widetilde{V}_{y}(U(\zeta_{m}))
\mathbf{1}_{\overline{\mathbb{S}}}(\zeta_{m})
\nonumber\\
&
-V_{y}(x)\mathbf{1}_{\{I(0)=0\}}-\widetilde{V}_{y}(x)\mathbf{1}_{\{I(0)=1\}}
\leq
\int_{0-}^{\zeta_{m}}\mathrm{e}^{-J(s)}\mathcal{G}(U,I,\mathbb{S}_{m},\overline{\mathbb{S}})(s-)\mathrm{d}s\notag\\
&+\int_{0-}^{\zeta_{m}} \mathrm{e}^{-J(s)}
\big[\sigma V_{y}^{\prime}(U(s-))
\mathbf{1}_{\mathbb{S}_{m}}(s)
\mathrm{d}B(s)+\widetilde{\sigma}\widetilde{V}_{y}^{\prime}(U(s-))
\mathbf{1}_{\overline{\mathbb{S}}}(s)
\mathrm{d}\widetilde{B}(s)\big]
\nonumber\\
&-\int_{0-}^{\zeta_{m}} \mathrm{e}^{-J(s)}V_{y}^{\prime}(U(s))\mathbf{1}_{\mathbb{S}_{m}}(s)\mathbf{1}_{U(s+)\geq c}\mathrm{d}C_{D}(s)
-\sum_{s\leq \zeta_{m}} \mathrm{e}^{-J(s)}\triangle D(s)\mathbf{1}_{\mathbb{S}_{m}}(s)
\mathbf{1}_{U(s+)\geq c}\nonumber\\
&+\sum_{s\leq \zeta_{m}} \mathrm{e}^{-J(s)}\big[V_{y}(U(s+))-V_{y}(U(s+)+\triangle D(s))+\triangle D(s)\big]\mathbf{1}_{\mathbb{S}_{m}}(s)\mathbf{1}_{U(s+)\geq c}
\nonumber\\
&-\int_{0-}^{\zeta_{m}}\int_{0}^{1} \mathrm{e}^{-J(s)}
\left(V_{y}^{\prime}(U(s-))\mathbf{1}_{\mathbb{S}_{m}}(s)z\overline{N}(\mathrm{d}s,\mathrm{d}z)+ \widetilde{V}_{y}^{\prime}(U(s-))
\mathbf{1}_{\overline{\mathbb{S}}}(s)
z\overline{\widetilde{N}}(\mathrm{d}s,\mathrm{d}z)\right)
\nonumber\\
&+\int_{0-}^{\zeta_{m}}\int_{0}^{\infty}\mathrm{e}^{-J(s)}[V_{y}(U(s-)
-z)-V_{y}(U(s-))+V_{y}^{\prime}(U(s-))z\mathbf{1}_{z\leq 1}]
\mathbf{1}_{\mathbb{S}_{m}}(s)
\overline{N}(\mathrm{d}s,\mathrm{d}z)
\nonumber\\
&+\int_{0-}^{\zeta_{m}}\int_{0}^{\infty}\mathrm{e}^{-J(s)}[\widetilde{V}_{y}(U(s-)
-z)-\widetilde{V}_{y}(U(s-))+\widetilde{V}_{y}^{\prime}(U(s-))z\mathbf{1}_{z\leq 1}]
\mathbf{1}_{\overline{\mathbb{S}}}(s)
\overline{\widetilde{N}}(\mathrm{d}s,\mathrm{d}z),
\end{align}
}for $x\in(-\infty,\infty)$, where 
$\Delta D(s)=D(s+)-D(s)$, $\mathbb{S}_{m}=\cup_{i\geq0}[\mathcal{T}_{2i},\zeta_{1/m}^{(2i+1)-})\subseteq \mathbb{S}$ when $I(0)=0$, and $\mathbb{S}_{m}=\cup_{i\geq1}[\mathcal{T}_{2i-1},\zeta_{1/m}^{(2i)-})\subseteq \mathbb{S}$ when $I(0)=1$.  In light of the fact that $V^{\prime}(x)\geq 1$ for all $x\in[c,\infty)$, we have
\begin{align}
\left[V_{y}(U(s+))-V_{y}(U(s+)+\triangle D(s))+ \triangle D(s)\right]\mathbf{1}_{\mathbb{S}_{m}}(s)\mathbf{1}_{\{U(s+)\geq c\}}
\le 0.\label{4.19}
\end{align}
Hence, by \eqref{hjbine4}, \eqref{4.18}, \eqref{4.19} and the facts that $V_{y}^{\prime}(x)\geq1$ for $x\geq c$, $U(t)>0$ for all $t\in \mathbb{S}_{m}$ as well as $U(t)<c$ for all $t\in \overline{\mathbb{S}}$, for $x\in(-\infty,\infty)$, we have that
{\small
\begin{align}\label{3.36}
&\mathrm{e}^{-J(\zeta_{m})}V_{y}\left(U(\zeta_{m})\right)\mathbf{1}_{\mathbb{S}}(\zeta_{m})
+\mathrm{e}^{-J(\zeta_{m})}\widetilde{V}_{y}\left(U(\zeta_{m})\right)
\mathbf{1}_{\overline{\mathbb{S}}}(\zeta_{m})
-V_{y}(x)\mathbf{1}_{\{I(0)=0\}}-\widetilde{V}_{y}(x)\mathbf{1}_{\{I(0)=1\}}
\nonumber\\
\leq& -\sum_{s\leq
\zeta_{m}
}\mathrm{e}^{-J(s)}\Delta D(s)\mathbf{1}_{\mathbb{S}_{m}}(s)\mathbf{1}_{\{U(s+)\geq c\}}-\int_{0-}^{\zeta_{m}}\mathrm{e}^{-J(s)}\mathbf{1}_{\mathbb{S}_{m}}(s)\mathbf{1}_{\{U(s+)\geq c\}}\mathrm{d}C_{D}(s)
\nonumber\\
&+\int_{0-}^{\zeta_{m}} \mathrm{e}^{-J(s)}\left(\sigma
V_{y}^{\prime}(U(s))\mathbf{1}_{\mathbb{S}_{m}}(s)
\,\mathrm{d}B(s)+\widetilde{\sigma}\widetilde{V}_{y}^{\prime}(U(s))
\mathbf{1}_{\overline{\mathbb{S}}}(s)
\,\mathrm{d}\widetilde{B}(s)\right)
\nonumber\\
&-\int_{0-}^{\zeta_{m}}\int_{0}^{1} \mathrm{e}^{-J(s)}
\left(V_{y}^{\prime}(U(s-))\mathbf{1}_{\mathbb{S}_{m}}(s)z\overline{N}(\mathrm{d}s,\mathrm{d}z)+ \widetilde{V}_{y}^{\prime}(U(s-))
\mathbf{1}_{\overline{\mathbb{S}}}(s)
z\overline{\widetilde{N}}(\mathrm{d}s,\mathrm{d}z)\right)
\nonumber\\
&+\int_{0-}^{\zeta_{m}}\int_{0}^{\infty}\mathrm{e}^{-J(s)}\big[V_{y}(U(s-)
-z)-V_{y}(U(s-))+V_{y}^{\prime}(U(s-))z\mathbf{1}_{z\leq1}\big]
\mathbf{1}_{\mathbb{S}_{m}}(s)
\,
\overline{N}(\mathrm{d}s,\mathrm{d}z)
\nonumber\\
&+\int_{0-}^{\zeta_{m}}\int_{0}^{\infty}\mathrm{e}^{-J(s)}\big[\widetilde{V}_{y}(U(s-)
-z)-\widetilde{V}_{y}(U(s-))+\widetilde{V}_{y}^{\prime}(U(s-))z\mathbf{1}_{z\leq1}\big]
\mathbf{1}_{\overline{\mathbb{S}}}(s)
\overline{\widetilde{N}}(\mathrm{d}s,\mathrm{d}z)
.
\end{align}
}We claim that the stochastic integral
$$\int_{0-}^{\zeta_{m}}\int_{0}^{1} \mathrm{e}^{-J(s)}
\left(V_{y}^{\prime}(U(s-))\mathbf{1}_{\mathbb{S}_{m}}(s)z\overline{N}(\mathrm{d}s,\mathrm{d}z)+ \widetilde{V}_{y}^{\prime}(U(s-))
\mathbf{1}_{\overline{\mathbb{S}}}(s)
z\overline{\widetilde{N}}(\mathrm{d}s,\mathrm{d}z)\right),$$
has zero mean. Indeed, its expectation under $\mathbb{P}_{x,0}$ (the computation under $\mathbb{P}_{x,1}$ is similar and hence omitted) can be expressed by
{\small\begin{eqnarray}
&&
\mathbb{E}_{x,0} \Big[ \int_{0-}^{\zeta_{m}} \int_{0}^{1} \mathrm{e}^{-J(s)}
[V_{y}^{\prime}(U(s-))\mathbf{1}_{\mathbb{S}_{m}}(s)z\overline{N}(\mathrm{d}s,\mathrm{d}z)+ \widetilde{V}_{y}^{\prime}(U(s-))
\mathbf{1}_{\overline{\mathbb{S}}}(s)
z\overline{\widetilde{N}}(\mathrm{d}s,\mathrm{d}z)] \Big]
\nonumber\\
&=&
\sum_{i\geq 0}\mathbb{E}_{x,0}\Big[\mathbf{1}_{\{\mathcal{T}_{2i}<\zeta_{m}\}}\mathbb{E}_{x,0}\Big[
\int_{\mathcal{T}_{2i}}
^{\zeta_{m}\wedge \zeta_{1/m}^{(2i+1)-}}\int_{0}^{1} \mathrm{e}^{-J(s)}
V_{y}^{\prime}(U(s-))z\overline{N}(\mathrm{d}s,\mathrm{d}z)\Big|
\mathcal{F}_{\mathcal{T}_{2i}}\Big] \Big]
\nonumber\\
&&+
\sum_{i\geq 0} \mathbb{E}_{x,0} \Big[\mathbf{1}_{\{\mathcal{T}_{2i+1}<\zeta_{m}\}}
\mathbb{E}_{x,0} \Big[
\int_{\mathcal{T}_{2i+1}}^{\zeta_{m}\wedge\mathcal{T}_{2i+2}} \int_{0}^{1} \mathrm{e}^{-J(s)}
\widetilde{V}_{y}^{\prime}(U(s-))
z\overline{\widetilde{N}}(\mathrm{d}s,\mathrm{d}z)\Big|\mathcal{F}_{\mathcal{T}_{2i+1}}
 \Big]
 \Big]
\nonumber\\
&=&
\sum_{i\geq 0}\mathbb{E}_{x,0}\Big[\mathrm{e}^{-J(\mathcal{T}_{2i})}\mathbf{1}_{\{\mathcal{T}_{2i}<\zeta_{m}\}}
\mathbb{E}_{x,0}\Big[
\int_{0}
^{\xi_{1/m,m}^{(2i+1)}}
\nonumber\\
&&\quad \quad \quad
\times\int_{0}^{1} \mathrm{e}^{-qs}
V_{y}^{\prime}(U_{2i+1}(s-))z\overline{N}(\mathcal{T}_{2i}+\mathrm{d}s,\mathrm{d}z)\Big|
\mathcal{F}_{\mathcal{T}_{2i}}\Big]\Big]
\nonumber\\
&&+
\sum_{i\geq 0}\mathbb{E}_{x,0}\Big[\mathrm{e}^{-J(\mathcal{T}_{2i+1})}
\mathbf{1}_{\{\mathcal{T}_{2i+1}<\zeta_{m}\}}\mathbb{E}_{x,0}\Big[
\int_{0}^{\xi_{-m,c}^{(2i+2)}}
\nonumber\\
&&\quad \quad \quad \times
\int_{0}^{1} \mathrm{e}^{-(q+\lambda)s}
\widetilde{V}_{y}^{\prime}(\widetilde{U}_{2i+2}(s-))
z\overline{\widetilde{N}}(\mathcal{T}_{2i+1}+\mathrm{d}s,\mathrm{d}z)\Big|
\mathcal{F}_{\mathcal{T}_{2i+1}}\Big]\Big],
\label{3.29}
\end{eqnarray}}where the two inner conditional expectations on the right hand side of \eqref{3.29} equal to $0$ because of the L\'evy-It\^{o} decomposition theorem (see, Theorem 2.1 in \cite{Kyp2014}) and three facts: (i) $V_{y}^{\prime}(U_{2i+1}(s-))$ is uniformly bounded as $1/m< U_{2i+1}(s-)<m$ for all $s<\xi_{1/m,m}^{(2i+1)}$; (ii) $\widetilde{V}_{y}^{\prime}(\widetilde{U}_{2i+2}(s-))$ is uniformly bounded because $-m<\widetilde{U}_{2i+2}(s-)<c$ for all $s<\xi_{-m,c}^{(2i+2)}$; and (iii) 
$\{\widetilde{X}(\mathcal{T}_{2i+1}+s)-\widetilde{X}(\mathcal{T}_{2i+1});s\geq0\}$ (resp., $\{X(\mathcal{T}_{2i}+s)-X(\mathcal{T}_{2i});s\geq0\}$) is independent of $\mathcal{F}_{\mathcal{T}_{2i+1}}$ (resp., $\mathcal{F}_{\mathcal{T}_{2i}}$) as well as $\widetilde{N}(\mathcal{T}_{2i+1}+s,\mathrm{d}z)-\widetilde{N}(\mathcal{T}_{2i+1},\mathrm{d}z)$
(resp., $N(\mathcal{T}_{2i}+s,\mathrm{d}z)-N(\mathcal{T}_{2i},\mathrm{d}z)$) inducing the compensated Poisson random measure $\overline{\widetilde{N}}(\mathcal{T}_{2i+1}+\mathrm{d}s,\mathrm{d}z)$ (resp., $\overline{N}(\mathcal{T}_{2i}+\mathrm{d}s,\mathrm{d}z)$) is independent of $\mathcal{F}_{\mathcal{T}_{2i+1}}$ (resp., $\mathcal{F}_{\mathcal{T}_{2i}}$).
We also claim that the next two integrals
{\small
\begin{align*}
&\int_{0-}^{\zeta_{m}}\int_{0}^{\infty}\mathrm{e}^{-J(s)}\left(V_{y}(U(s-)
-z)-V_{y}(U(s-))+V_{y}^{\prime}(U(s-))z\mathbf{1}_{z\leq1}\right)
\mathbf{1}_{\mathbb{S}_{m}}(s)
\,
\overline{N}(\mathrm{d}s,\mathrm{d}z),
\\
&
\int_{0-}^{\zeta_{m}}\int_{0}^{\infty}\mathrm{e}^{-J(s)}\left(\widetilde{V}_{y}(U(s-)
-z)-\widetilde{V}_{y}(U(s-))+\widetilde{V}_{y}^{\prime}(U(s-))z\mathbf{1}_{z\leq1}\right)
\mathbf{1}_{\overline{\mathbb{S}}}(s)
\,\,
\overline{\widetilde{N}}(\mathrm{d}s,\mathrm{d}z),
\end{align*}}have zero mean. Here, we only show that the expectation of the first stochastic integral under $\mathbb{P}_{x,0}$ is $0$. In fact, similar to the arguments in handling \eqref{3.29}, we have
{\small\begin{eqnarray}
&&
\mathbb{E}_{x,0} \Big[
\int_{0-}^{\zeta_{m}}\int_{0}^{\infty}\mathrm{e}^{-J(s)}[V_{y}(U(s-)
-z)-V_{y}(U(s-))+V_{y}^{\prime}(U(s-))z\mathbf{1}_{z\leq1}]
\mathbf{1}_{\mathbb{S}_{m}}(s)
\,
\overline{N}(\mathrm{d}s,\mathrm{d}z)
 \Big]
\nonumber\\
&=&
\sum_{i\geq 0}\mathbb{E}_{x,0}\Big[\mathrm{e}^{-J(\mathcal{T}_{2i})}\mathbf{1}_{\{\mathcal{T}_{2i}<\zeta_{m}\}}\mathbb{E}_{x,0}\Big[
\int_{0}
^{\xi_{1/m,m}^{(2i+1)}}\int_{0}^{\infty}\mathrm{e}^{-qs}
\nonumber\\
&&
\times[V_{y}(U_{2i+1}(s-)
-z)-V_{y}(U_{2i+1}(s-))+V_{y}^{\prime}(U_{2i+1}(s-))z\mathbf{1}_{z\leq1}]
\overline{N}(\mathcal{T}_{2i}+\mathrm{d}s,\mathrm{d}z)\Big|
\mathcal{F}_{\mathcal{T}_{2i}}\Big] \Big]
\nonumber\\
&=&0,\nonumber
\end{eqnarray}}where we have used Corollary 4.6 in \cite{Kyp2014}. To be precise, one key condition of Corollary 4.6 in \cite{Kyp2014} is that the expected integral of the integrand in the above inner integral with respect to the L\'evy measure $\upsilon$ is finite, which is already verified in Lemma \ref{lem3.4} thanks to the fact $1/m<U_{2i+1}(s-)<m$ for every $s<\xi_{1/m,m}^{(2i+1)}$. Other conditions of Corollary 4.6 in \cite{Kyp2014} including the process in the integrand of the above inner integral being left-continuous can all be checked easily.
Similarly, it holds that the integral
$$\int_{0-}^{\zeta_{m}} \mathrm{e}^{-J(s)}\big(\sigma
V_{y}^{\prime}(U(s-))\mathbf{1}_{\mathbb{S}_{m}}(s)
\mathrm{d}B(s)+\widetilde{\sigma}
\widetilde{V}_{y}^{\prime}(U(s-))\mathbf{1}_{\overline{\mathbb{S}}}(s)
\mathrm{d}\widetilde{B}(s)\big),$$
also has zero mean. Here, we only show the expectation of the above integral under $\mathbb{P}_{x,0}$ is $0$. Actually, we have
\begin{eqnarray}
&&
\mathbb{E}_{x,0} \Big[ \int_{0-}^{\zeta_{m}} \mathrm{e}^{-J(s)}\big[\sigma
V_{y}^{\prime}(U(s-))\mathbf{1}_{\mathbb{S}_{m}}(s)
\mathrm{d}B(s)+\widetilde{\sigma}
\widetilde{V}_{y}^{\prime}(U(s-))\mathbf{1}_{\overline{\mathbb{S}}}(s)
\mathrm{d}\widetilde{B}(s)\big] \Big]
\nonumber\\
&=&
\sum_{i\geq 0}\mathbb{E}_{x,0}\Big[\mathrm{e}^{-J(\mathcal{T}_{2i})}\mathbf{1}_{\{\mathcal{T}_{2i}<\zeta_{m}\}}
\mathbb{E}_{x,0}\Big[
\int_{0}
^{\xi_{1/m,m}^{(2i+1)}}
\sigma \mathrm{e}^{-qs} V_{y}^{\prime}(U_{2i+1}(s-))
\nonumber\\
&&
\times \mathrm{d}B(\mathcal{T}_{2i}+s)\Big|
\mathcal{F}_{\mathcal{T}_{2i}}\Big]\Big]
+
\sum_{i\geq 0}\mathbb{E}_{x,0}\Big[\mathrm{e}^{-J(\mathcal{T}_{2i+1})}
\mathbf{1}_{\{\mathcal{T}_{2i+1}<\zeta_{m}\}}\mathbb{E}_{x,0}\Big[
\int_{0}^{\xi_{-m,c}^{(2i+2)}}
\nonumber\\
&&\quad\quad\quad\times
\widetilde{\sigma} \mathrm{e}^{-(q+\lambda)s}\widetilde{V}_{y}^{\prime}(\widetilde{U}_{2i+2}(s-))\mathrm{d}\widetilde{B}(\mathcal{T}_{2i+1}+s)\Big|
\mathcal{F}_{\mathcal{T}_{2i+1}}\Big]\Big]
,
\nonumber
\label{3.30}
\end{eqnarray}
which is $0$ due to the arguments on Page 146 in \cite{KarS1991}, and the facts that the integrands are uniformly bounded and $B(\mathcal{T}_{2i}+s)-B(\mathcal{T}_{2i})$ (reps., $\widetilde{B}(\mathcal{T}_{2i+1}+s)-\widetilde{B}(\mathcal{T}_{2i+1})$) is a standard Brownian motion independent of $\mathcal{F}_{\mathcal{T}_{2i}}$ (resp., $\mathcal{F}_{\mathcal{T}_{2i+1}}$). 
Taking expectations on both sides of \eqref{3.36} and recalling the no-negativity of $V_{y}(x)$ and $\widetilde{V}_{y}(x)$, and $\mathbf{1}_{\mathbb{S}_{m}}(s)\uparrow\mathbf{1}_{\mathbb{S}}(s)$ a.s. as $m\uparrow\infty$, we have that
\begin{align}\label{3.37}
&V_{y}(x)\mathbf{1}_{\{I(0)=0\}}+\widetilde{V}_{y}(x)\mathbf{1}_{\{I(0)=1\}}
\nonumber\\
\geq&
\mathbb{E}_{x,I(0)}\Big[\int_{0-}^{\zeta_{m}} \mathrm{e}^{-J(s)}\mathbf{1}_{\mathbb{S}_{m}}(s)\mathbf{1}_{\{U(s+)\geq c\}}\mathrm{d}C_{D}(s)\Big]\notag\\
&+ \mathbb{E}_{x,I(0)}\Big[\sum_{s\leq \zeta_{m}}\mathrm{e}^{-J(s)}( D(s+)-D(s))\mathbf{1}_{\mathbb{S}_{m}}(s)\mathbf{1}_{\{U(s+)\geq c\}}\Big]
\nonumber\\
\rightarrow&
\mathbb{E}_{x,I(0)}\Big[\int_{0-}^{\infty} \mathrm{e}^{-J(s)}\mathbf{1}_{\{U(s+)\geq c\}}\mathrm{d}C_{D}(s)\Big]\notag\\
&+ \mathbb{E}_{x,I(0)}\Big[\sum_{s<\infty}\mathrm{e}^{-J(s)}( D(s+)-D(s))\mathbf{1}_{\{U(s+)\geq c\}}\Big]
\nonumber\\
=&
\mathbb{E}_{x,I(0)}\Big[\int_{0-}^{\infty} \mathrm{e}^{-J(s)}\mathbf{1}_{\{U(s+)\geq c\}}\,\mathrm{d}D(s)\Big]
\nonumber\\
=&
V_{D}(x)\mathbf{1}_{\{I(0)=0\}}+\widetilde{V}_{D}(x)\mathbf{1}_{\{I(0)=1\}},\quad x\in(-\infty,\infty),\,m\rightarrow\infty,\,t\rightarrow\infty,
\end{align}
where the last equality follows by the Poisson method introduced in \cite{LiZ2014} and the memory-less property of exponential random variables. To wit, we denote by $(T_{i})_{i\geq 1}$ the successive arrival times of a Poisson processes with rate $\lambda$ that are independent of the process $(X(t),\widetilde{X}(t),U(t), I(t))_{t\geq 0}$, and denote by $\mathcal{F}_{X,\widetilde{X},U,I}$ the smallest sigma field generated by it. It holds that
{\small\begin{align}
&\mathbb{E}_{x,I(0)}\Big[\int_{0-}^{\infty} \mathrm{e}^{-J(t)}\mathbf{1}_{\{U(t+)\geq c\}}\,\mathrm{d}D(t)\Big]
\nonumber\\
=&
\mathbb{E}_{x,I(0)}\Big[\int_{0-}^{\infty} \mathrm{e}^{-qt}\,\mathbb{E}\Big[\left.\mathbf{1}_{\{((T_{i})_{i\geq 1})\,\cap\,\overline{\mathbb{S}}\,\cap\,[0,t]=\emptyset\}}\right|\mathcal{F}_{X,\widetilde{X},U,I}\Big]\mathbf{1}_{\{U(t+)\geq c\}}\,\mathrm{d}D(t)\Big]
\nonumber\\
=&
\mathbb{E}_{x,I(0)}\Big[\,\mathbb{E}\Big[\int_{0-}^{\infty} \mathrm{e}^{-qt}\mathbf{1}_{\{T_{D}>t\}}\mathbf{1}_{\{U(t+)\geq c\}}\,\mathrm{d}D(t)\Big|\mathcal{F}_{X,\widetilde{X},U,I}\Big]\Big]
\nonumber\\
=&
\mathbb{E}_{x,I(0)}\Big[\int_{0-}^{T_{D}} \mathrm{e}^{-qt}\mathbf{1}_{\{U(t+)\geq c\}}\,\mathrm{d}D(t)\Big]
\nonumber\\
=&
V_{D}(x)\mathbf{1}_{\{I(0)=0\}}+\widetilde{V}_{D}(x)\mathbf{1}_{\{I(0)=1\}}, \quad x\in(-\infty,\infty).\nonumber
\end{align}}By \eqref{3.37} and the arbitrariness of $D$, 
we can conclude that
$V_{y}(x)\geq \sup\limits_{D\in\mathcal{D}}V_{D}(x)$ for all $x\in\left(0,\infty\right)$, and $\widetilde{V}_{y}(x)\geq \sup\limits_{D\in\mathcal{D}}\widetilde{V}_{D}(x)$ for all $x\in\left(-\infty,c\right)$.
\end{proof}

By Theorem \ref{Lem4-1}, to prove the optimality of the barrier dividend strategy with the barrier $d^{*}$ defined in \eqref{def.d*}, it only remains to verify that $\mathcal{A}V_{d^{*}}(x)-q V_{d^{*}}(x)\leq 0$ for almost every $x\in\left(0,\infty\right)$, that $\widetilde{\mathcal{A}}\widetilde{V}_{d^{*}}(x)-\left(q+\lambda\right) \widetilde{V}_{d^{*}}(x)\leq 0$ for almost every $x\in\left(-\infty,c\right)$, that $V^{\prime}(x)\geq 1$ for almost every $x\in\left[c,\infty\right)$, and that both $V_{d^{*}}(x)$ and $\widetilde{V}_{d^{*}}(x)$ are non-decreasing.

Similar to the proof of Lemma 4.2 in \cite{KypRS2010} or a method in Section 5 of \cite{AvPP2015}, we have the next result.

\begin{lem}\label{lem4-2}
Let $d^{*}$ and $V_{y}(x)$ be defined in \eqref{def.d*} and \eqref{V_{y}(x)}. For any $y\in\left[d^{*},\infty\right)$ and any $x\in\left(0,y\right)$, we have that
\begin{align}\label{ver.xles.a.01}
\mathcal{A}V_{y}(x)-q V_{y}(x)=0.
\end{align}
In particular, it holds that
$$\mathcal{A}V_{d^{*}}(x)-q V_{d^{*}}(x)=0, \quad x\in\left(0,d^{*}\right).$$
\end{lem}

\begin{proof}
By Equation (5.1), Definition 5.1, as well as Corollary 5.9 in \cite{AvPP2015}, it follows that
$(\mathrm{e}^{-q(t\wedge \tau_{y}^{+}\wedge \tau_{0}^{-})}W_{q}(X(t\wedge \tau_{y}^{+}\wedge \tau_{0}^{-})))_{t\geq 0}$ and $(\mathrm{e}^{-q(t\wedge \tau_{y}^{+}\wedge \tau_{0}^{-})}Z_{q}(X(t\wedge \tau_{y}^{+}\wedge \tau_{0}^{-}), \widetilde{\Phi}_{q+\lambda}))_{t\geq 0}$ are two martingales. Recall that $\ell^{(q,\lambda)}_{c}(x)$ is a linear combination of $W_{q}(x)$ and $Z_{q}(x, \widetilde{\Phi}_{q+\lambda})$ (see \eqref{ell.def.}). As a result, for arbitrary $(\nu_{1},\nu_{2})\subseteq(0,y)$, the stopped process $\{\mathrm{e}^{-q(t\wedge \tau_{\nu_{2}}^{+}\wedge \tau_{\nu_{1}}^{-})}\ell^{(q,\lambda)}_{c}(X(t\wedge \tau_{\nu_{2}}^{+}\wedge \tau_{\nu_{1}}^{-}));\, t\geq 0\}$ is a martingale.
We claim that this martingale property implies
\eqref{ver.xles.a.01}.
Indeed, Meyer-It\^{o} formula gives
\begin{align}
&\mathrm{e}^{-q(t\wedge \tau_{\nu_{2}}^{+}\wedge \tau_{\nu_{1}}^{-})}\ell^{(q,\lambda)}_{c}\left(X(t\wedge \tau_{\nu_{2}}^{+}\wedge \tau_{\nu_{1}}^{-})\right)-\ell^{(q,\lambda)}_{c}\left(x\right)\nonumber\\
=&
\int_{0-}^{t\wedge \tau_{\nu_{2}}^{+}\wedge \tau_{\nu_{1}}^{-}}\mathrm{e}^{-q s}(\mathcal{A}-q)\ell^{(q,\lambda)}_{c}(X(s-))\mathrm{d}s\notag\\
&+\int_{0-}^{t\wedge \tau_{\nu_{2}}^{+}\wedge \tau_{\nu_{1}}^{-}} \mathrm{e}^{-q s}\ell^{(q,\lambda)\prime}_{c}(X(s-))\Big[\sigma\mathrm{d}B(s)-\int_{0}^{1}z\overline{N}(\mathrm{d}s,\mathrm{d}z)\Big]\nonumber\\
&+\int_{0-}^{t\wedge \tau_{\nu_{2}}^{+}\wedge \tau_{\nu_{1}}^{-}}\int_{0}^{\infty}\mathrm{e}^{-q s}\Big[\ell^{(q,\lambda)}_{c}\left(X(s-)-z\right)-\ell^{(q,\lambda)}_{c}\left(X(s-)\right)\notag\\
&+\ell^{(q,\lambda)\prime}_{c}(X(s-))z\mathbf{1}_{(0,1]}(z)\Big]
\overline{N}(\mathrm{d}s,\mathrm{d}z),\, t\geq0.\nonumber
\end{align}
By Lemma \ref{lem3.4}, $(\mathcal{A}-q)\ell^{(q,\lambda)}_{c}(x)$ is bounded on $[\nu_{1},\nu_{2}]$ and is continuous in $x\in[\nu_{1},\nu_{2}]$.
Hence, taking expectations on both sides of the above equality and using the dominated convergence theorem by sending $t$ to $+\infty$, we arrive at
\begin{align}
0&=\frac{1}{q}\mathbb{E}_{x}\Big[\int_{0-}^{\infty}q\mathrm{e}^{-q s}(\mathcal{A}-q)\ell^{(q,\lambda)}_{c}\left(X(s)\right)\mathbf{1}_{\{s<\tau_{\nu_{2}}^{+}\wedge \tau_{\nu_{1}}^{-}\}}\mathrm{d}s\Big]
\nonumber\\
&=\int_{\nu_{1}}^{\nu_{2}}(\mathcal{A}-q)\ell^{(q,\lambda)}_{c}(\omega)
\frac{1}{q}\mathbb{P}_{x}
\left(X_{e_{q}}\in\mathrm{d}\omega,e_{q}<\tau_{\nu_{2}}^{+}\wedge \tau_{\nu_{1}}^{-}\right).\label{eq-0-0}
\end{align}
Together with the arbitrariness of $(\nu_{1},\nu_{2})$, the continuity of  $(\mathcal{A}-q)\ell^{(q,\lambda)}_{c}(x)$ in $x\in[\nu_{1},\nu_{2}]$, and the fact that the $q$-potential measure $\frac{1}{q}\mathbb{P}_{x}
\left(X_{e_{q}}\in\mathrm{d}\omega,e_{q}<\tau_{\nu_{2}}^{+}\wedge \tau_{\nu_{1}}^{-}\right)$ has a strictly positive resolvent density on $(\nu_{1},\nu_{2})$ (see the proof of Lemma 4.2 in \cite{KypRS2010}), the equality in \eqref{eq-0-0} implies the desired result (\ref{ver.xles.a.01}) for $x\in\left(0,y\right)$.
\end{proof}

Moreover, we have the next auxiliary result.
\begin{lem}\label{lem4-3}
Suppose that 
 the tail of the L\'evy measure $\upsilon$ is log-convex. Let $d^{*}$ and $V_{y}(x)$ be defined in \eqref{def.d*} and \eqref{V_{y}(x)}, respectively. We have that $V_{d^{*}}^{\prime}(x)\geq 1$ for $x\in[c,\infty)$, and
\begin{align}\label{ver.xles.a}
\mathcal{A}V_{d^{*}}(x)-q V_{d^{*}}(x)\leq 0, \quad x\in(d^{*}, \infty).
\end{align}
\end{lem}

\begin{proof}
By the definitions of $d^{*}$ in \eqref{def.d*} and $V_{d^{*}}$ in \eqref{V_{y}(x)}, it is straightforward to see that $V_{d^{*}}^{\prime}(x)\geq 1$ for $x\in[c,\infty)$.
It remains to prove \eqref{ver.xles.a}.
First, thanks to \eqref{V_{y}(x)} 
and (\ref{ver.xles.a.01}), it holds that, for $x\in(d^{*},\infty)$
\begin{align}\label{3.25}
0=&\lim\limits_{v\uparrow x}\big[\mathcal{A}V_{x}(v)-q V_{x}(v)\big]=
\gamma+\frac{\sigma^{2}}{2}\frac{\ell^{(q,\lambda)\prime\prime}_{c}(x)}
{\ell^{(q,\lambda)\prime}_{c}(x)}\nonumber\\
&+\int_{0}^{\infty}
\Big[\frac{\ell^{(q,\lambda)}_{c}(x-z)}{\ell^{(q,\lambda)\prime}_{c}(x)}
-\frac{\ell^{(q,\lambda)}_{c}(x)}{\ell^{(q,\lambda)\prime}_{c}(x)}
+z\mathbf{1}_{(0,1)}(z)\Big]
\upsilon(\mathrm{d}z)-q\frac{\ell^{(q,\lambda)}_{c}(x)}
{\ell^{(q,\lambda)\prime}_{c}(x)}.
\end{align}
On the other hand, using the definition \eqref{V_{y}(x)} with $y=d^{*}$, 
we get that, for $x\in(d^{*},\infty)$,
\begin{align}\label{3.26}
\mathcal{A}V_{d^{*}}(x)-q V_{d^{*}}(x)=&\gamma+\int_{x-d^{*}}^{\infty}
\Big[\frac{\ell^{(q,\lambda)}_{c}(x-z)}{\ell^{(q,\lambda)\prime}_{c}(d^{*})}
-\frac{\ell^{(q,\lambda)}_{c}(d^{*})}{\ell^{(q,\lambda)\prime}_{c}(d^{*})}\notag\\
&-(x-d^{*})
+z\mathbf{1}_{(0,1)}(z)\Big]\upsilon(\mathrm{d}z)
\nonumber\\
&+\int_{0}^{x-d^{*}}
\left(-z
+z\mathbf{1}_{(0,1)}(z)\right)\upsilon(\mathrm{d}z)-q\Big[\frac{\ell^{(q,\lambda)}_{c}(d^{*})}
{\ell^{(q,\lambda)\prime}_{c}(d^{*})}+(x-d^{*})\Big].
\end{align}
Lemma \ref{thm4-2} gives that $\ell^{(q,\lambda)\prime}_{c}$ is increasing on $[d^{*},\infty)$, which together with Remark \ref{rem01} implies that
\begin{align}\label{4.20}
\frac{\ell^{(q,\lambda)\prime\prime}_{c}(x)}
{\ell^{(q,\lambda)\prime}_{c}(x)}\geq 0,\quad x\in\left(d^{*},\infty\right),
\end{align}
when $\sigma\in(0,\infty)$. Here, note that if $\sigma>0$, $W_{q}(x)$ (hence, $\ell^{(q,\lambda)}_{c}(x)$) is twice continuously differentiable. By virtue of Lemma \ref{thm4-2} and the mean value theorem, for $x\in\left(d^{*},\infty\right)$, $z\in\left(0,x-d^{*}\right]$, we get that
\begin{align}\label{4.21}
&\Big[\frac{\ell^{(q,\lambda)}_{c}(x-z)}{\ell^{(q,\lambda)\prime}_{c}(x)}
-\frac{\ell^{(q,\lambda)}_{c}(x)}{\ell^{(q,\lambda)\prime}_{c}(x)}\Big]
-(-z)=z\left(1-\frac{\ell^{(q,\lambda)\prime}_{c}(\omega)}{\ell^{(q,\lambda)\prime}_{c}(x)}\right)
\geq 0,\quad
\end{align}
for some $\omega\in(x-z,x)\subseteq\left(d^{*},x\right)$, and
\begin{align}\label{sign.lim.oper.}
&\Big[\frac{\ell^{(q,\lambda)}_{c}(x-z)}{\ell^{(q,\lambda)\prime}_{c}(x)}
-\frac{\ell^{(q,\lambda)}_{c}(x)}{\ell^{(q,\lambda)\prime}_{c}(x)}\Big]
-\Big[\frac{\ell^{(q,\lambda)}_{c}(x-z)}{\ell^{(q,\lambda)\prime}_{c}(d^{*})}
-\frac{\ell^{(q,\lambda)}_{c}(d^{*})}{\ell^{(q,\lambda)\prime}_{c}(d^{*})}-(x-d^{*})
\Big]
\nonumber\\
=&
\ell^{(q,\lambda)}_{c}(x-z)\Big[\frac{1}{\ell^{(q,\lambda)\prime}_{c}(x)}
-\frac{1}{\ell^{(q,\lambda)\prime}_{c}(d^{*})}\Big]
+\frac{\ell^{(q,\lambda)}_{c}(d^{*})}{\ell^{(q,\lambda)\prime}_{c}(d^{*})}\notag\\
&-\frac{\ell^{(q,\lambda)}_{c}(d^{*})+\ell^{(q,\lambda)\prime}_{c}(\omega)(x-d^{*})}
{\ell^{(q,\lambda)\prime}_{c}(x)}
+\big(x-d^{*}\big)
\nonumber\\
=&
\Big[\ell^{(q,\lambda)}_{c}(d^{*})-\ell^{(q,\lambda)}_{c}(x-z)\Big]\Big[\frac{1}
{\ell^{(q,\lambda)\prime}_{c}(d^{*})}-
\frac{1}{\ell^{(q,\lambda)\prime}_{c}(x)}
\Big]
\nonumber\\
&+\big(x-d^{*}\big)\Big[1-\frac{\ell^{(q,\lambda)\prime}_{c}(\omega)}
{\ell^{(q,\lambda)\prime}_{c}(x)}\Big]
\geq0,\quad x\in(d^{*},\infty),\,z\in(x-d^{*},\infty),
\end{align}
for some $\omega\in\left(d^{*},x\right)$, and
\begin{align}\label{4.26}
&-\frac{\ell^{(q,\lambda)}_{c}(x)}
{\ell^{(q,\lambda)\prime}_{c}(x)}-\Big[
-\frac{\ell^{(q,\lambda)}_{c}(d^{*})}
{\ell^{(q,\lambda)\prime}_{c}(d^{*})}-(x-d^{*})\Big]\notag\\
=&
\left(x-d^{*}\right)\Big[1-\frac{\ell^{(q,\lambda)\prime}_{c}(\omega)}
{\ell^{(q,\lambda)\prime}_{c}(x)}\Big]
+\Big[\frac{\ell^{(q,\lambda)}_{c}(d^{*})}
{\ell^{(q,\lambda)\prime}_{c}(d^{*})}-\frac{\ell^{(q,\lambda)}_{c}(d^{*})}
{\ell^{(q,\lambda)\prime}_{c}(x)}\Big]\geq0, \quad x\in\left(d^{*},\infty\right),
\end{align}
for some $\omega\in\left(d^{*},x\right)$.
From \eqref{4.20}, \eqref{4.21}, \eqref{sign.lim.oper.} and \eqref{4.26}, it follows that the right hand side of \eqref{3.26} is less than the right hand side of \eqref{3.25} with the latter being zero. Therefore,  \eqref{ver.xles.a} is verified for $x\in\left(d^{*},\infty\right)$.
\end{proof}

\begin{lem}\label{lem4-4}
Let $d^{*}$ and $\widetilde{V}_{y}(x)$ be defined in \eqref{def.d*} and \eqref{tildeV_{y}(x)}, respectively. For $x\in\left(-\infty,c\right)$, we have that
\begin{align}\label{ver.xles.a.001}
\widetilde{\mathcal{A}}\widetilde{V}_{d^{*}}(x)-(q+\lambda) \widetilde{V}_{d^{*}}(x)=0.
\end{align}
\end{lem}
\begin{proof}
It follows from definition that $\widetilde{V}_{d^{*}}(x)=\frac
{\mathrm{e}^{\widetilde{\Phi}_{q+\lambda}(x-c)}}
{\ell^{(q,\lambda)\prime}_{c}(d^{*})}$, $x\in\left(-\infty, c\right)$.
Hence, we have that
\begin{align}
&
\widetilde{\mathcal{A}}\widetilde{V}_{d^{*}}(x)-(q+\lambda) \widetilde{V}_{d^{*}}(x)=
\frac{\mathrm{e}^{\widetilde{\Phi}_{q+\lambda}(x-c)}}{\ell^{(q,\lambda)\prime}_{c}(d^{*})}
\Big[\widetilde{\gamma} \widetilde{\Phi}_{q+\lambda}+\frac{1}{2}\widetilde{\sigma}^{2}\widetilde{\Phi}_{q+\lambda}^{2}\nonumber\\
&+\int_{(0,\infty)}\left(\mathrm{e}^{-\widetilde{\Phi}_{q+\lambda}z}-1+\widetilde{\Phi}_{q+\lambda} z\mathbf{1}_{(0,1)}(z)\right)\widetilde{\upsilon}(\mathrm{d}z)-(q+\lambda)\Big]
\nonumber\\
=&
\frac{\mathrm{e}^{\widetilde{\Phi}_{q+\lambda}(x-c)}}{\ell^{(q,\lambda)\prime}_{c}(d^{*})}
\left(\widetilde{\psi}\left(\widetilde{\Phi}_{q+\lambda}\right)-q-\lambda\right)=0,\quad x\in\left(-\infty, c\right),
\nonumber
\end{align}
which is the desired result \eqref{ver.xles.a.001}.
\end{proof}

We are ready to present and prove the main result of this section.

\begin{thm}\label{HJB4}
Suppose that the tail of the L\'evy measure $\upsilon$ is log-convex. If either $\psi(\widetilde{\Phi}_{q+\lambda})\geq q$ or $W_{q}(0+)=0$ \emph{(}i.e., either $\sigma>0$ or $\int_{(0,1)}z\upsilon(\mathrm{d}z)=\infty$\emph{)},
then the barrier dividend strategy with $d^{*}$ defined in \eqref{def.d*} is the optimal singular dividend control attaining the maximal value function under the Chapter 11 bankruptcy.
\end{thm}

\begin{proof}
It is straightforward to see that $\widetilde{V}_{d^{*}}(x)$ is non-decreasing on $(-\infty,\infty)$ by \eqref{tildeV_{y}(x)}. Moreover, thanks to \eqref{V_{y}(x)} and Remark \ref{rem01}, $V_{d^{*}}(x)$ is increasing on $(-\infty,\infty)$ if and only if $V_{d^*}(0)<V_{d^*}(0+)$, which holds if and only if $\psi(\widetilde{\Phi}_{q+\lambda})\geq q$ or $W_{q}(0+)=0$ (i.e., $X$ has paths of unbounded variation). The conclusion of Theorem \ref{HJB4} is a direct consequence of Theorem \ref{Lem4-1} and Lemmas \ref{lem4-2}-\ref{lem4-4}.
\end{proof}

\begin{rem}
\label{3.2.amo}
It is worth noting that the solution to our optimal control problem \eqref{op.str.} relies only on the log-convexity assumption on the L\'evy measure of $X$ (rather than on that of $\widetilde{X}$). Indeed, this is a natural consequence of the fact that the function $\ell_{c}^{(q,\lambda)}(x)$ defined by \eqref{ell.def.}, and hence, the candidate optimal value function $V_{d^{*}}(x)$ and $\widetilde{V}_{d^{*}}(x)$ given by \eqref{def.d*} and \eqref{V_{y}(x)}-\eqref{tildeV_{y}(x)}, all depend on the L\'evy triplet $(\widetilde{\gamma}, \widetilde{\sigma}, \widetilde{\upsilon})$ of $\widetilde{X}$ only through the
single parameter $\widetilde{\Phi}_{q+\lambda}$. In particular, $\widetilde{V}_{d^{*}}(x)$ has a very simple form such that the associated Hamilton-Jacobi-Bellman (HJB) equation in terms of $(\widetilde{\gamma}, \widetilde{\sigma}, \widetilde{\upsilon})$ is readily satisfied; see, Lemma \ref{lem4-4}.
\end{rem}

\begin{rem}\label{rem3.2}
There is one scenario that Theorem \ref{HJB4} is not applicable, namely, the case when
$W_{q}(0+)>0$ \emph{(}i.e., $\sigma=0$, $\int_{(0,1)}z\upsilon(\mathrm{d}z)<\infty$, and $W_{q}(0+)=1/(\gamma+\int_{(0,1)}z\upsilon(\mathrm{d}z))$\emph{)} and $\psi(\widetilde{\Phi}_{q+\lambda})<q$ (i.e., $X\neq\widetilde{X}$). Under this scenario, it will be unwise to set the safety barrier $c$ to be positive. 
Otherwise, the insurer would rather choose not to run the business. 
In fact, in this case, one can verify that
\begin{eqnarray}
\lim\limits_{\triangle t\downarrow0}\frac{\mathbb{E}\big[\mathrm{e}^{\widetilde{\Phi}_{q+\lambda}\cdot X(t+\triangle t)}\big]-\mathbb{E}\big[\mathrm{e}^{\widetilde{\Phi}_{q+\lambda}\cdot X(t)}\big]}{\triangle t}
 <
\lim\limits_{\triangle t\downarrow0}\frac{\mathbb{E}\big[\mathrm{e}^{\widetilde{\Phi}_{q+\lambda}\cdot \widetilde{X}(t+\triangle t)}\big]-\mathbb{E}\big[\mathrm{e}^{\widetilde{\Phi}_{q+\lambda}\cdot \widetilde{X}(t)}\big]}{\triangle t},\nonumber
\end{eqnarray}
where the left \emph{(}resp., right\emph{)} hand side can be interpreted as the instant average amount of growth \emph{(}in the sense of exponential moments\emph{)} of the surplus process at solvent \emph{(}resp., insolvent\emph{)} times.
In view of the above inequality, the insurer may prefer staying in the insolvent state to quickly accumulate surplus instead of running the business in the solvency state with a slower surplus growth. To resolve this dilemma and to reduce the long term costly interventions, the regulator may need to set the safety barrier in the way that $c\downarrow0$ \emph{(}suppose $U$ is still well defined\emph{)}. The 
candidate optimal barrier defined in \eqref{def.d*} is 
reduced to
\begin{align}
d^{*}=\sup\{d\geq 0: Z_{q}^{\prime}(x,\widetilde{\Phi}_{q+\lambda})\geq Z_{q}^{\prime}(d,\widetilde{\Phi}_{q+\lambda})\text{ \emph{for all} }x\geq 0\}=0,\nonumber
\end{align}
where we have used by definition that 
$$Z_{q}^{\prime\prime}(x,\widetilde{\Phi}_{q+\lambda})
=\widetilde{\Phi}_{q+\lambda}^{2} Z_{q}(x,\widetilde{\Phi}_{q+\lambda})
-\big[\psi(\widetilde{\Phi}_{q+\lambda})-q\big]\big[\widetilde{\Phi}_{q+\lambda}W_{q}(x)
+W_{q}^{\prime}(x)\big]>0,\quad x\geq0.$$
In addition, the 
functions 
given by \eqref{V_{y}(x)} and \eqref{tildeV_{y}(x)} with $y=0$ is reduced as
\begin{eqnarray}
V_{0}(x)=\widetilde{V}_{0}(x)=x\mathbf{1}_{\{x\geq0\}}+\frac{\mathbf{1}_{\{x\geq0\}}+\mathrm{e}^{\widetilde{\Phi}_{q+\lambda} x}\mathbf{1}_{\{x<0\}}}{\widetilde{\Phi}_{q+\lambda}-(\psi(\widetilde{\Phi}_{q+\lambda})-q) W_{q}(0+)}.\nonumber
\end{eqnarray}
By arguments similar to those in Lemmas \ref{lem4-3} and \ref{lem4-4}, it is easy to verify that
$(\widetilde{\mathcal{A}}-(q+\lambda) )\widetilde{V}_{0}(x)=0$ for all $x<0$
and
$(\mathcal{A}-q )V_{0}(x)\leq0$ for all $x>0$, implying that the barrier dividend strategy with barrier level $0$ is indeed the optimal dividend strategy.
Therefore, $\psi(\widetilde{\Phi}_{q+\lambda})<q$ and $W_{q}(0+)>0$ leads to an extreme case that can be reasonably ruled out in the real life practice.
\end{rem}

\begin{rem}\label{rem3.3}
In the special case when $\widetilde{X}\equiv X$, we actually have $\psi(\widetilde{\Phi}_{q+\lambda})=\psi(\Phi_{q+\lambda})=q+\lambda>q$. In fact, the optimal dividend problem with $X\equiv\widetilde{X}$ and $c=0$ has already been addressed by \cite{Re19}, and our Theorem \ref{HJB4} covers the result in \cite{Re19}. By Theorem \ref{HJB4}, the optimal barrier level of the optimal dividend strategy is equal to
\begin{align}
d^{*}=\sup\{d\geq c: \hbar^{(q,\lambda)\prime}_{c}(x)\geq \hbar^{(q,\lambda)\prime}_{c}(d)\mbox{ \emph{for all} }x\geq c\},\nonumber
\end{align}
where
$
\hbar_{c}^{(q,\lambda)}(x)
=\lambda \frac{W_{q}(x)}{W_{q}(c)}\int_{0}^{c}\mathrm{e}^{-\Phi_{q+\lambda} w}W_{q}(w)\mathrm{d}w+
\mathrm{e}^{-\Phi_{q+\lambda}c}Z_{q}(x,\Phi_{q+\lambda})$.
Letting $c\downarrow0$, we can simplify $d^{*}$ to
\begin{align}\label{d*red.}
d^{*}:=\sup\{d\geq 0: Z_{q}^{\prime}(x,\Phi_{q+\lambda})\geq Z_{q}^{\prime}(d,\Phi_{q+\lambda})\mbox{ \emph{for all} }x\geq 0\},
\end{align}
where we used the fact that $
\lim_{c\downarrow0}\hbar_{c}^{(q,\lambda)}(x)
=Z_{q}(x,\Phi_{q+\lambda})$. Note that we may have $d^*>0$. Moreover, the value functions under the barrier dividend strategy with the barrier $d^{*}$ given by \eqref{d*red.} satisfy
\begin{eqnarray}\label{V_{y}(x).c=0}
V_{d^{*}}(x)=\widetilde{V}_{d^{*}}(x)=
\frac
{Z_{q}(x,\Phi_{q+\lambda})}
{Z_{q}^{\prime}(d^{*},\Phi_{q+\lambda})}\mathbf{1}_{\{x\leq d^{*}\}}
+\Big[\frac
{Z_{q}(d^{*},\Phi_{q+\lambda})}
{Z_{q}^{\prime}(d^{*},\Phi_{q+\lambda})}+x-d^{*}\Big]\mathbf{1}_{\{x>d^{*}\}}.
\nonumber
\end{eqnarray}
It is easy to verify that
$(\mathcal{A}-(q+\lambda) )\widetilde{V}_{d^{*}}(x)=0$ for all $x<0$, $(\mathcal{A}-q )V_{d^{*}}(x)=0$ for all $x\in(0,d^{*})$,
and
$(\mathcal{A}-q)V_{d^{*}}(x)\leq0$ for all $x>d^{*}$ since 
$Z_{q}^{\prime}(x,\Phi_{q+\lambda})$ is non-decreasing on $[d^{*},\infty)$.
It follows that the barrier dividend strategy with $d^{*}$ given in \eqref{d*red.} is indeed optimal, and therefore, Theorem \ref{HJB4} is consistent with the special case  $\widetilde{X}\equiv X$ and $c=0$ studied in \cite{Re19}.
\end{rem}

\section{An Illustrative Example}\label{sec:example}

Theorem \ref{HJB4} shows that $d^{*}$ defined in \eqref {def.d*} is indeed the optimal barrier under some mild conditions. We now carry out explicit computations to identify $d^*$ in an example of Cram\'er-Lundberg process $X$ with exponential jump sizes, namely, a process $X$ defined by a deterministic drift $p$ (the premium income) subtracting a compound Poisson process with jump intensity $\lambda_{0}$ and exponentially distributed jump sizes with mean $1/\mu$. Hence, the process $X$ has paths of bounded variation. In addition, the scale function of $X$ reads as $$W_{q}(x)=p^{-1}(A_{+}\mathrm{e}^{q_{+}x}-A_{-}\mathrm{e}^{q_{-}x}),$$
where $A_{\pm}:=(\mu+q_{\pm})/(q_{+}-q_{-})$ and $q_{\pm}:=(q+\lambda_{0}-\mu p\pm \sqrt{(q+\lambda_{0}-\mu p)^{2}+4pq\mu})/2p$ (i.e., $q_{+}>0$, $q_{-}<0$). In particular, $W_{q}(0+)=p^{-1}>0$.
\begin{itemize}
    \item 
For the first case $\psi(\widetilde{\Phi}_{q+\lambda})>q$, to derive the explicit expression of $d^{*}$, we note that this, together with \eqref{ell.def.} and some tedious algebraic manipulation, yields that
\begin{align}
\ell_{c}^{(q,\lambda)}(x)
&=(\psi(\widetilde{\Phi}_{q+\lambda})-q)
p^{-1}(B_{+}\mathrm{e}^{q_{+}x}/q_{+}-B_{-}\mathrm{e}^{q_{-}x}/q_{-}),
\nonumber\\
\ell^{(q,\lambda)\prime}_{c}(x)&=
(\psi(\widetilde{\Phi}_{q+\lambda})-q)
p^{-1}(B_{+}\mathrm{e}^{q_{+}x}-B_{-}\mathrm{e}^{q_{-}x}),\nonumber
\end{align}
where
$B_{\pm}:=\frac{p^{-1}A_{\pm}q_{\pm}\big[\mu+\widetilde{\Phi}_{q+\lambda}-(\mu+q_{\mp})\mathrm{e}^{(q_{\mp}-\widetilde{\Phi}_{q+\lambda})c} \big]}{(q_{+}-\widetilde{\Phi}_{q+\lambda})(q_{-}-\widetilde{\Phi}_{q+\lambda})}$
with $B_{+}>0$ and $B_{-}<0$.

When $\psi(\widetilde{\Phi}_{q+\lambda})>q$, we obtain that the optimal barrier $d^*=c\vee (\ln (B_{-}q_{-}/B_{+}q_{+})/(q_{+}-q_{-}))$.
On one hand, we remark that the model when $d^*>c$, i.e. $c<\ln (B_{-}q_{-}/B_{+}q_{+})/(q_{+}-q_{-})$, is the case of interest from the practical point of view because the insurance company pays dividend whenever it is in the solvency state and the surplus level is more than adequate to attain a high barrier $d^*>c$. On the other hand, the extreme case may happen that $d^*=c$, and the insurance company needs to pay dividend whenever the surplus comes back to the solvency barrier $c$. Note that the switch to solvency state is triggered immediately when the barrier $c$ is hit, hence this extreme case does not change the fact the insurance company will switch between solvency and insolvency states until Chapter 11 bankruptcy occurs. However, the surplus level can never climb above a prescribed solvency level $c$, which not only gives very low incentives for insurer to run the business, but may also cause the regulator more frequent interventions.
We can see that $d^*=c$ 
happens when the regulator is too conservative and sets $c$ too high (i.e. $c\geq \ln (B_{-}q_{-}/B_{+}q_{+})/(q_{+}-q_{-})$). Consequently, from the regulator's perspective,  one message to take from this example is that the regulator may monitor the dividend barriers from all market participants to actively adjust the level of the safety barrier $c$ so that some unnecessary costs in long term and large scale interventions can be avoided.

\item For the second case $\psi(\widetilde{\Phi}_{q+\lambda})=q$, it is easy to see that $Z_{q}(x,\widetilde{\Phi}_{q+\lambda})=\mathrm{e}^{\widetilde{\Phi}_{q+\lambda}x}$. Hence, it holds that
\begin{align}
\ell_{c}^{(q,\lambda)}(x)=
\mathrm{e}^{-\widetilde{\Phi}_{q+\lambda} c}W_{q}(c)
\mathrm{e}^{\widetilde{\Phi}_{q+\lambda} x},
\quad
\ell_{c}^{(q,\lambda)\prime}(x)=
\widetilde{\Phi}_{q+\lambda}\mathrm{e}^{-\widetilde{\Phi}_{q+\lambda} c}W_{q}(c)
\mathrm{e}^{\widetilde{\Phi}_{q+\lambda} x}.
\nonumber
\end{align}
It then follows that the optimal barrier $d^{*}=c$. 

\item For the final case $\psi(\widetilde{\Phi}_{q+\lambda})<q$, it is straightforward to check that 
\begin{align}
    \ell_{c}^{(q,\lambda)}(x)&=K_{+}\mathrm{e}^{q_{+}x}/q_{+}-K_{-}\mathrm{e}^{q_{-}x}/q_{-}+K_{0}\mathrm{e}^{\widetilde{\Phi}_{q+\lambda}x},
    \nonumber
\end{align}
where $K_{\pm}=\psi(\widetilde{\Phi}_{q+\lambda})-q)B_{\pm}/p$ and 
$K_{0}=\big[1-
\frac{(\psi(\widetilde{\Phi}_{q+\lambda})-q)(\mu+\widetilde{\Phi}_{q+\lambda}) }{p(q_{+}-\widetilde{\Phi}_{q+\lambda})(q_{-}-\widetilde{\Phi}_{q+\lambda})}\big]\big[\frac{A_{+}\mathrm{e}^{q_{+}c}}{p}-
\frac{A_{-}\mathrm{e}^{q_{-}c}}{p}\big]
\mathrm{e}^{-\widetilde{\Phi}_{q+\lambda}c}
$. Theorem \ref{HJB4} is not applicable to handle this case (recall that $W_{q}(0+)>0$). However, by Remark \ref{rem3.2}, this case
corresponds to some extreme behavior that can be excluded in practice because it occurs when $\widetilde{X}$ under the regulator's intervention grows upward even faster than $X$ on average so that both the safety barrier and the optimal dividend barrier will turn out to be $0$.
\end{itemize}

\ \\
\ \\
\textbf{Acknowledgements}: {\small
Wenyuan Wang acknowledges the financial support from the National Natural Science Foundation of China (No.12171405; No.11661074) and the Program for New Century Excellent Talents in Fujian Province University. Xiang Yu acknowledges the financial support from the Hong Kong Polytechnic University research grant (No.P0031417). Xiaowen Zhou acknowledges the financial support from  NSERC (RGPIN-2021-04100) and National Natural Science Foundation of China (No.11771018; No.12171405).}
\ \\
\ \\
\textbf{Conflict of interest}\\
The authors declare that they have no conflict of interest.

\end{document}